\documentclass[12pt]{article}
\usepackage{amsmath}
\usepackage{amsmath,amssymb, amsthm, latexsym, amsfonts, epsfig, color,graphicx,}
\usepackage{mathrsfs}
\usepackage{indentfirst}
\usepackage{chngcntr}
\usepackage{hyperref}

\hypersetup{colorlinks=true,linkcolor=black}

\begin{document}
	\renewcommand{\a}{\alpha}
	\newcommand{\D}{\Delta}
	\newcommand{\ddt}{\frac{d}{dt}}
	\counterwithin{equation}{section}
	\newcommand{\e}{\epsilon}
	\newcommand{\eps}{\varepsilon}
	\newtheorem{theorem}{Theorem}[section]
	
	\newtheorem{proposition}{Proposition}[section]
	\newtheorem{lemma}[proposition]{Lemma}
	\newtheorem{remark}{Remark}[section]
	\newtheorem{example}{Example}[section]
	\newtheorem{definition}{Definition}[section]
	\newtheorem{corollary}{Corollary}[section]
	\makeatletter
	\newcommand{\rmnum}[1]{\romannumeral #1}
	\newcommand{\Rmnum}[1]{\expandafter\@slowromancap\romannumeral #1@}
	\makeatother
	
	\title{Scattering for the fractional magnetic Schr\"odinger operators \footnote{This work was supported by National Natural Science Foundation of China (61671009, 12171178).}}	
	
	\author{Lei Wei\, \, \, Zhiwen Duan$^{\star}$\\
		{\small {\it  School of Mathematics and Statistics, Huazhong University of Science }}\\
		{\small {\it and Technology, Wuhan, {\rm 430074,} P.R.China.}}  \\
		{\small {\it $^{\star}$Corresponding author. Email: duanzhw@hust.edu.cn }}\\
		{\small {\it Contributing author. Email: weileimath@hust.edu.cn}}}
	\date{}
	\maketitle
	
	\centerline{\large\bf Abstract}	
	\par	
	In this paper, we prove the existence of the scattering operator for the fractional magnetic Schr\"odinger operators. For this, we construct the fractional distorted Fourier transforms with magnetic potentials. Applying the properties of the distorted Fourier transforms, the existence and the asymptotic completeness of the wave operators are obtained. Furthermore, we prove the absence of positive eigenvalues for Schr\"odinger operators.

	
	
	\noindent{\bf Authors’ Contributions:} All authors contributed equally to this work. All authors read and approved the final manuscript.\\
	
	%\noindent{\bf Funding Statement:} This work was supported by National Science Foundation of China (NSFC) (61671009,12171178).\\
	
	\noindent{\bf Conflicts of Interest:} The authors declare that they have no conflicts of interest to report regarding the present study.

	\section{Introduction}
	In this paper, we define  	$\mathscr{H}$ is the self-adjoint realization of $(-\Delta_{A})^{\frac{s}{2}}$ for $0<s<2$ in $L^{2}(\mathbb{R}^{n})$, where the magnetic Schr\"odinger operators $-\Delta_{A}$ is described as follows: $$-\Delta_{A}:=-(\nabla+iA)^{2}=-\Delta-iA\cdot\nabla-i\nabla\cdot A+\arrowvert A\arrowvert^{2},$$ and  $A(x)=(A_{1}(x),...,A_{n}(x))\in C^{2}$ is a real-valued vector function which satisfies the following hypothesis:
	\begin{equation}\label{1.1}
		\arrowvert A(x)\arrowvert+\arrowvert\nabla A(x)\arrowvert+\arrowvert\nabla\nabla A(x)\arrowvert\leqslant C\langle x \rangle^{-\beta},
	\end{equation}
	for some $\beta>n+2$.
	
	In addition, we state the properties and results of the operator spectrum. By direct calculation for any $0\neq u\in \mathscr{S}$
	\begin{equation}\label{1.2}
		(-(\nabla+iA)^{2}u,u)_{L^{2}}
		=\int_{\mathbb{R}^{n}}\sum\limits_{j=1}^{n}\arrowvert(\partial_{x_{j}}+iA_{j})u\arrowvert^{2}dx>0,
	\end{equation}
 we have $\sigma(-\Delta_{A})\subset[0,\infty)$.
From \cite{KT} (also see Lemma A.2 in \cite{KK}), we know that the singular continuous spectrum of $-\Delta_{A}$ is  empty, i.e., $\sigma_{sc}(-\Delta_{A})=\emptyset$, and there is no positive eigenvalue, i.e., $\sigma_{pp}(-\Delta_{A})=\emptyset$. In this paper, we also assume that zero is neither an eigenvalue nor a resonance of the magnetic Schr\"odinger operator $-\Delta_{A}$, where zero is called a resonance if there exists a distribution solution of $-\Delta_{A}u=0$ such that $u\in L^{2,-\sigma}$ but  $u\notin L^{2}$ for any $\sigma>\frac{1}{2}$ in  $n\geqslant3$. This definition for resonance differs from the case of dimension $n=2$ in which $u$ lies in $L^{p}$ for some $p\in(2,\infty]$.
	
	\begin{remark}\label{remark1.1}
		For $n=1$, we recall the following fact: let $G(x)$ be a real smooth function and $\mathcal{T}$ be the gauge transformation:
		$$\mathcal{T}u(t,x)=e^{-iG(x)}u(t,x).$$
		Then $\mathcal{T}$ transforms $(-\Delta_{A})u$ into the same one with $A$ being replaced by $A+\nabla_{x}G$. In particular, taking $G(x)=-\int_{0}^{x}A(y)dy$, the potential $A$ is eliminated from $(-\Delta_{A})u$ by the gauge transformation $\mathcal{T}$. Therefore, we only consider the case where the space dimension $n\geqslant2$.
	\end{remark}
		
	In the well known case $-\Delta$, the term \textit{short range} physically means potential $V(x)$ decays fast enough to ensure that a quantum scattering system has locality in a large scale. For example, Agmon \cite{Ag} and H\"{o}rmander \cite{H} proved the spectral and scattering theory for Schr\"{o}dinger operators $\Delta_{V}$ when $V(x)$ satisfies a short range condition, i.e. $V(x)$ decays at infinity like $\sim \arrowvert x\arrowvert^{-1-\varepsilon},\varepsilon>0$. In the recent survey, Schlag \cite{Sc2} established a structure formula for the intertwining wave operators in $\mathbb{R}^{3}$, which as applied to spectral theory. 
	
	It is then natural to ask for the short range scattering theory of $(-\Delta)^\frac s2$ when $s>0$, and such study has already drawn some authors' attention, while understanding the dynamics of fractional Schr\"{o}dinger equations is of physical importance, for example, $\sqrt{-\Delta}+V$ is the relativistic Hamiltonian with mass $0$. Specifically, Giere \cite{g} considered functions of $-\Delta$ and obtained asymptotic completeness based on semigroup difference method, which applies to $(-\Delta)^\frac s2$ in the case $0<s<2$ perturbed by potentials decaying at the rate $(1+|x|)^{-1-\varepsilon}$ in dimension $n>2+2\varepsilon-s$. Kitada \cite{K} considered $s\geqslant1$ and obtained through eigenfunction expansion method that the asymptotic completeness for short range potentials verifying $|V(x)|\leqslant C(1+|x|)^{-1-\varepsilon}$. For the fractional Schr\"{o}dinger operators $(-\Delta)^{\frac{s}{2}}+V$ ($s>0$) and $V(x)$ is a real-valued short range potential, the scattering theory in a Banach space $B\subset L^{2}(\mathbb{R}^{n})$ was studied in \cite{ZHZ}.
	
	For the magnetic Schr\"{o}dinger operators, Demuth and Ouhabaz \cite{DO} proved the existence and completeness of the wave operators $W_{\pm}(A(b),-\Delta)$ for Schr\"{o}dinger operators $A(b)=-(\nabla-ib(x))^{2}$, $x\in\mathbb{R}^{n}$. In addition, the Strichartz and smoothing estimates for the magnetic Schr\"{o}dinger operators were studied, Erdo\v{g}an, Goldberg and Schlag \cite{EGS} proved the Strichartz and smoothing estimates for the Schr\"{o}dinger operators with almost critical magnetic potentials in higher dimensions. D'Ancona, Fanelli, Vega and Visciglia \cite{DFVV} obtained the endpoint Strichartz estimates for the magnetic Schr\"{o}dinger equations. The other results of the Strichartz and smoothing estimates for the magnetic Schr\"{o}dinger operators can be found in \cite{DF,GST}.

	$Outline\ of\ this\ paper:$
	
	In Section 2,  we construct the fractional magnetic distorted Fourier transforms as following
	$$F^{A}_{\pm}u(\xi)=\mathscr{F}((I+V_{x}R^{s}_{0}(\lambda\pm i0))^{-1}u)(\xi),\ s>0.$$ In order to make this construct meaningful, we prove the boundedness of operators $V_{x}$ and $R^{s}_{0}(\lambda\pm i0)$ (see Subsection 2.1 and Subsection 2.2 respectively).
	First, the integral representation of $V_{x}:=(-\Delta_{A})^{\frac{s}{2}}-(-\Delta)^{\frac{s}{2}}$ is the key, i.e. for $0<s<2$, $$V_{x}=c(s)\int^{\infty}_{0}\tau^{\frac{s}{2}}(\tau-\Delta_{A})^{-1}V_{1}(x)(\tau-\Delta)^{-1}d\tau,$$ 
	where $V_{1}(x):=-iA\cdot\nabla-i\nabla\cdot A+\arrowvert A\arrowvert^{2}$. Applying this integral representation, we obtain that $V_{x}$ is a short range perturbation of $\mathscr{H}_{0}$, where $\mathscr{H}_{0}$ is the self-adjoint realization of $(-\Delta)^{\frac{s}{2}}$ for any $s>0$ in $L^{2}(\mathbb{R}^{n})$. Next, we prove some estimates of fractional free resolvent operators $$R^{s}_{0}(\lambda\pm i0)=((-\Delta)^{\frac{s}{2}}-(\lambda\pm i0))^{-1}.$$ Specifically, we establish the bounds of $R^{s}_{0}(\lambda\pm i0)$ from positive weighted space $L^{2,\sigma}$ into negative weighted space $H^{s,-\sigma}$ with $\sigma>\frac{1}{2}$, $s>0$. Therefore, combining these bounds of operators $V_{x}$ and $R^{s}_{0}(\lambda\pm i0)$, the distorted Fourier transforms are well defined.

	In Section 3, we consider the scattering behavior of $(-\Delta_{A})^{\frac{s}{2}}$, thus we study the properties of wave operator. On the one hand, combining some properties of magnetic Shr\"{o}dinger operator and $V_{x}$ is a short range perturbation of $\mathscr{H}_{0}$, the existence of the wave operators is derived. On the other hand, by establishing the relationship between wave operators and distorted Fourier transforms: $$F^{A}_{\pm}W_{\pm}=\mathscr{F},$$ the asymptotic completeness of the wave operators is obtained. 
	
In Section 4, we prove the absence of positive eigenvalues for $(-\Delta_{A})^{\frac{s}{2}}$ with $0<s<2$. Inspired by the well-known fact in \cite{KT}, which is the Shr\"{o}dinger operator $-\Delta_{A}$ has no positive eigenvalues, we have the main result of this section.

Based on all the conclusions, we have the following theorem:
	
	\begin{theorem}\label{theorem1.1}	
		Let $A(x)$ satisfy hypothesis \eqref{1.1}, $\mathscr{H}_\mathrm{ac}$ and $\mathscr{H}_\mathrm{sc}$ respectively denote the absolutely continuous and singular continuous subspaces of $L^2(\mathbb{R}^n)$ with respect to $\mathscr{H}$. Then the wave operators $W_{\pm}$ defined by the strong $L^{2}$ limits
		\begin{equation}\label{1.3}
			W_{\pm}u=\lim_{t\rightarrow\pm\infty}e^{it\mathscr{H}}e^{-it\mathscr{H}_{0}}u,\ u\in L^{2}(\mathbb{R}^{n}),
		\end{equation}
		are asymptotically complete, that is 	
		\begin{equation}\label{1.4}	
			\mathrm{Ran}(W_+)=\mathrm{Ran}(W_-)=\mathscr{H}_\mathrm{ac}\quad\text{and}\quad\mathscr{H}_\mathrm{sc}=\emptyset.
		\end{equation}
		Consequently, the scattering operator $S=W_+^*W_-$ is unitary.	
	\end{theorem}
	
\begin{remark}	
In our other article, the application of Theorem \ref{theorem1.1} is crucial for obtaining dispersive estimates for the solutions of nonlinear magnetic Schrödinger equations. These dispersive estimates play a significant role in understanding the asymptotic behavior of Schrödinger operators with magnetic potentials. They provide valuable insights into the spreading properties, scattering behavior, and long-term dynamics of the nonlinear equations.

The dispersive estimates obtained through the application of Theorem \ref{theorem1.1} have wide-ranging implications in various fields, including quantum mechanics, nonlinear partial differential equations (PDEs), and related areas. They contribute to our understanding of the behavior of solutions to nonlinear magnetic Schrödinger equations, shedding light on the effects of magnetic potentials on the dynamics of the system.

Moreover, the dynamics of fractional Schrödinger operators, which are studied in the paper, are of particular importance due to their relevance to anomalous diffusion, fractal structures, disordered media, and non-local interactions. These operators provide a mathematical framework for describing physical phenomena that exhibit fractional or non-local behavior. The analysis of their dynamics is a fascinating area of research in mathematical physics and quantum mechanics, contributing to the understanding of complex systems and their mathematical properties.
\end{remark}

\bigskip

	\phantomsection
	\addcontentsline{toc}{section}{Notation}
	\hspace{-5mm}\textbf{Notation}:
		
		\begin{itemize}
			\setlength{\baselineskip}{1.5\baselineskip}
			
			\item $\mathscr{S}(\mathbb{R}^{n})$ is Schwartz space, i.e. the set of all real- or complex-valued $C^{\infty}$ functions on $\mathbb{R}^{n}$ such that for every nonnegative integer $\kappa$ and every multi-index $\alpha$, $\sup\limits_{x\in\mathbb{R}^{n}}(1+\arrowvert x\arrowvert^{2})^{\frac{\kappa}{2}}\arrowvert D^{\alpha}u(x)\arrowvert<\infty$; $\mathscr{S}'(\mathbb{R}^{n})$ is space of tempered distribution on $\mathbb{R}^{n}$, i.e. the topological dual of $\mathscr{S}(\mathbb{R}^{n})$.

		\item $\mathscr{F}u(x)=\hat{u}=(2\pi)^{-\frac{n}{2}}\int_{\mathbb{R}^{n}}e^{-ix\cdot \xi}u(\xi)d\xi$.
		
		\item $(\mathscr{F}^{-1}u)(x)=\check{u}=(2\pi)^{-\frac{n}{2}}\int_{\mathbb{R}^{n}}e^{ix\cdot \xi}u(\xi)d\xi$.
		
		\item
		$H^{s}(\mathbb{R}^{n})=\{u\in L^{2}(\mathbb{R}^{n}):(1+\arrowvert\xi\arrowvert^{2})^{\frac{s}{2}}\hat{u}\in L^{2}(\mathbb{R}^{n})\}$.

		\item $L^{2,\sigma}(\mathbb{R}^{n})=\{u(x):(1+\arrowvert x\arrowvert^{2})^{\frac{\sigma}{2}}u(x)\in L^{2}(\mathbb{R}^{n})\}$.
		
		\item
		$H^{s,\sigma}(\mathbb{R}^{n})=\{u(x):(I-\Delta)^{\frac{s}{2}}u\in L^{2,\sigma}(\mathbb{R}^{n}), \forall s,\sigma\in\mathbb{R}\}.$

		\item
		$R_{0}(\tau)=(\tau-\Delta)^{-1}$,  $R_{0}^{s}(\tau)=(\tau+(-\Delta)^{\frac{s}{2}})^{-1}$ for $\tau>0$, $s>0$.
		
		\item
		$R_{A}(\tau)=(\tau-\Delta_{A})^{-1}$, $R_{A}^{s}(\tau)=(\tau+(-\Delta_{A})^{\frac{s}{2}})^{-1}$ for $\tau>0$, $s>0$.
	
		\item
		$F_\pm^{A}u(\xi)=\mathscr{F}((I+V_{x}R_0^{s}(\lambda\pm i0 ))^{-1}u)(\xi)$ almost everywhere in $M_\lambda=\{\xi\in\mathbb{R}^n;~\arrowvert\xi\arrowvert^s=\lambda\}$, where  $V_{x}=(-\Delta_{A})^{\frac{s}{2}}-(-\Delta)^{\frac{s}{2}}$, $s>0$.
		
		\item
		$\mathscr{H}_{0}$ is the self-adjoint realization of $(-\Delta)^{\frac{s}{2}}$ for any $s>0$ in $L^{2}(\mathbb{R}^{n})$.
		
		\item
		$\mathscr{H}$ is the self-adjoint realization of $(-\Delta_{A})^{\frac{s}{2}}$ for any $s>0$ in $L^{2}(\mathbb{R}^{n})$.
		
		\item
		$\sigma_{pp}$ is the pure point spectrum, $\sigma_{sc}$ is the singular continuous spectrum,  $\sigma_{ac}$ is the absolutely continuous spectrum.
		
		\item
		$D(-\Delta_{A})$ means the domain of the operator $-\Delta_{A}$.
		
	\end{itemize}

	\section{Construction of the distorted Fourier transforms at spectral points}
	In this section, we construct the fractional magnetic distorted Fourier transforms at spectral points $$F_\pm^{A}u(\xi):=\mathscr{F}((I+V_{x}R_0^{s}(\lambda\pm i0 ))^{-1}u)(\xi)$$
	of fractional magnetic Schr\"{o}dinger operators $(-\Delta_{A})^{\frac{s}{2}}$ for any $0<s<2$. This tool is very useful because it can generate some important properties for which need to be used to prove the existence and asymptotic completeness of the wave operators.
	
	\subsection{Short range perturbations}
	To make this structure meaningful, we should obtain property of the operator $(I+V_{x}R_0^{s}(\lambda\pm i0 ))^{-1}$. First, note $V_{x}:=(-\Delta_{A})^{\frac{s}{2}}-(-\Delta)^{\frac{s}{2}}$, we prove that $V_{x}$ is a short range perturbation of $\mathscr{H}_{0}$, i.e. $V_{x}$ is a bounded and compact operator from $H^{s,-\sigma}$ into $L^{2,\sigma}$ for some $\sigma>\frac{1}{2}$ and any $0<s<2$.
	
	\begin{lemma}\label{lemma2.1}
		Let $A(x)$ satisfy hypothesis \eqref{1.1}, then for $0<s<2$, we have
		\begin{flalign}
			&(-\Delta_{A})^{\frac{s}{2}}u-(-\Delta)^{\frac{s}{2}}u\nonumber\\
			=&c(s)\int^{\infty}_{0}\tau^{\frac{s}{2}}(\tau-\Delta_{A})^{-1}V_{1}(x)(\tau-\Delta)^{-1}ud\tau\nonumber\\
			=&c(s)\int^{\infty}_{0}\tau^{\frac{s}{2}}(\tau-\Delta)^{-1}V_{1}(x)(\tau-\Delta_{A})^{-1}ud\tau,\ \forall u\in D(-\Delta_{A}), \label{2.1}
		\end{flalign}
		where $V_{1}(x):=-iA\cdot\nabla-i\nabla\cdot A+\arrowvert A\arrowvert^{2}$ and  $(c(s))^{-1}=\int^{\infty}_{0}\tau^{\frac{s}{2}-1}(\tau+1)^{-1}d\tau$.		
	\end{lemma}
	\begin{proof}
		Let us recall the formula	
		\begin{equation*}
			(-\Delta_{A})^{\frac{s}{2}}=c(s)(-\Delta_{A})\int^{\infty}_{0}\tau^{\frac{s}{2}-1}(\tau-\Delta_{A})^{-1}d\tau,
		\end{equation*}
		where $(c(s))^{-1}=\int^{\infty}_{0}\tau^{\frac{s}{2}-1}(\tau+1)^{-1}d\tau$ with $0<s<2$ (also can see \cite{I}). Applying the integral representation above, we deduce that
		\begin{flalign*}
			&(-\Delta_{A})^{\frac{s}{2}}u=c(s)(-\Delta_{A})\int^{\infty}_{0}\tau^{\frac{s}{2}-1}(\tau-\Delta_{A})^{-1}ud\tau\\
			=&c(s)(-\Delta_{A})\int^{\infty}_{0}\tau^{\frac{s}{2}-1}((\tau-\Delta_{A})^{-1}-(\tau-\Delta)^{-1})ud\tau\\
			&+c(s)(-\Delta_{A})\int^{\infty}_{0}\tau^{\frac{s}{2}-1}(\tau-\Delta)^{-1}ud\tau\\
		=&-c(s)\int^{\infty}_{0}\tau^{\frac{s}{2}-1}(\arrowvert A \arrowvert^{2}-iA\cdot\nabla-i\nabla\cdot A)(\tau-\Delta)^{-1}ud\tau\\
			&+c(s)\int^{\infty}_{0}\tau^{\frac{s}{2}}(\tau-\Delta_{A})^{-1}(\arrowvert A \arrowvert^{2}-iA\cdot\nabla-i\nabla\cdot A)(\tau-\Delta)^{-1}ud\tau\\
			&+c(s)(-\Delta)\int^{\infty}_{0}\tau^{\frac{s}{2}-1}(\tau-\Delta)^{-1}ud\tau\\
			&+c(s)\int^{\infty}_{0}\tau^{\frac{s}{2}-1}(\arrowvert A \arrowvert^{2}-iA\cdot\nabla-i\nabla\cdot A)(\tau-\Delta)^{-1}ud\tau\\
			=&(-\Delta)^{\frac{s}{2}}u+c(s)\int^{\infty}_{0}\tau^{\frac{s}{2}}(\tau-\Delta_{A})^{-1}(\arrowvert A \arrowvert^{2}-iA\cdot\nabla-i\nabla\cdot A)(\tau-\Delta)^{-1}ud\tau.
		\end{flalign*}
		That is,
		\begin{flalign}
			&(-\Delta_{A})^{\frac{s}{2}}u-(-\Delta)^{\frac{s}{2}}u\nonumber\\
			=&c(s)\int^{\infty}_{0}\tau^{\frac{s}{2}}(\tau-\Delta_{A})^{-1}(\arrowvert A \arrowvert^{2}-iA\cdot\nabla-i\nabla\cdot A)(\tau-\Delta)^{-1}ud\tau\nonumber\\
			=&c(s)\int^{\infty}_{0}\tau^{\frac{s}{2}}(\tau-\Delta_{A})^{-1}V_{1}(x)(\tau-\Delta)^{-1}ud\tau.\label{2.2}
		\end{flalign}
		
		Because $\Delta_{A}$ and $\Delta$ have the same position in the expression. Using the formula
		\begin{equation*}
			(-\Delta)^{\frac{s}{2}}=c(s)(-\Delta)\int^{\infty}_{0}\tau^{\frac{s}{2}-1}(\tau-\Delta)^{-1}d\tau
		\end{equation*}
		and the similar method above, we obtain
		\begin{flalign}
			&(-\Delta_{A})^{\frac{s}{2}}u-(-\Delta)^{\frac{s}{2}}u\nonumber\\
			=&c(s)\int^{\infty}_{0}\tau^{\frac{s}{2}}(\tau-\Delta)^{-1}V_{1}(x)(\tau-\Delta_{A})^{-1}ud\tau.\label{2.3}
		\end{flalign}
		
		Combining \eqref{2.2} and \eqref{2.3}, \eqref{2.1} is proved. Hence the proof of Lemma \ref{lemma2.1} is completed.
	\end{proof}
	
	\begin{definition}\label{definition2.1}
		An operator $V_{x}$ is said to be a short range perturbation of $\mathscr{H}_{0}$ if $V_{x}$ maps the unit ball in $H^{s,-\sigma}$ into a precompact subset of $L^{2,\sigma}$ with some $\sigma>\frac{1}{2}$.
	\end{definition}
	
	\begin{remark}\label{remark2.1}
		Let $R_{0}(\tau)=(\tau-\Delta)^{-1}$. Then we have the weighted resolvent estimates for $\tau>0$, $1\leqslant r\leqslant2$ and $N\geqslant0$
\begin{flalign*}
&\|R_{0}(\tau)u\|_{L^{r}}\leqslant C\tau^{-1}\|u\|_{L^{r}},\\
&\|\langle x\rangle ^{-N}R_{0}(\tau)u\|_{L^{r}}\leqslant C\tau^{-1+\frac{N}{2}}\|u\|_{L^{r}},\\
&\|R_{0}(\tau)\langle x\rangle ^{-N}u\|_{L^{r}}\leqslant C\tau^{-1+\frac{N}{2}}\|u\|_{L^{r}},\\
&\|\langle x\rangle ^{N_{1}}R_{0}(\tau)\langle x\rangle ^{-N}u\|_{L^{r}}\leqslant C\tau^{-1+\frac{N-N_{1}}{2}}\|u\|_{L^{r}},
\end{flalign*}
where $N_{1}\leqslant N$.
\end{remark}

	Applying the integral representation of $V_{x}$, Lemma \ref{lemma2.1} and Remark \ref{remark2.1} we obtain the boundedness and compactness of $V_{x}$ as follows.
	
	\begin{lemma}\label{lemma2.2}
		Let $A(x)$ satisfy hypothesis \eqref{1.1}, then $V_{x}$ is a short range perturbation of $\mathscr{H}_{0}$.
	\end{lemma}
	\begin{proof}
		First, we prove
		$$\langle x\rangle^{\delta}V_{x}: H^{s,-\sigma}\rightarrow H^{\alpha,-\sigma},$$
		is a bounded operator for $\delta>1$, $\sigma>\frac{1}{2}$ and $0<\alpha<\min\{1,s\}$.
		
		From the definition of space norm
		\begin{equation}\label{2.4}
			\|\langle x\rangle^{\delta}V_{x}u\|^{2}_{H^{\alpha,-\sigma}(\mathbb{R}^{n})}\leqslant\|\langle x\rangle^{\delta-\sigma}V_{x}u\|^{2}_{L^{2}}+\|\langle x\rangle^{\delta-\sigma}(-\Delta)^{\frac{\alpha}{2}}V_{x}u\|^{2}_{L^{2}},
		\end{equation}	
		we split proof into two parts. Let's recall \eqref{2.1} for $0<s<2$, we have
		\begin{flalign}
			V_{x}u=&((-\Delta_{A})^{\frac{s}{2}}-(-\Delta)^{\frac{s}{2}})u\nonumber\\
			=&c(s)\int^{\infty}_{0}\tau^{\frac{s}{2}}(\tau-\Delta_{A})^{-1}V_{1}(x)(\tau-\Delta)^{-1}ud\tau\nonumber\\
			=&c(s)\int^{\infty}_{0}\tau^{\frac{s}{2}}(\tau-\Delta)^{-1}V_{1}(x)(\tau-\Delta_{A})^{-1}ud\tau,\label{2.5}
		\end{flalign}
		where $V_{1}(x):=-iA\cdot\nabla-i\nabla\cdot A+\arrowvert A\arrowvert^{2}$ and  $(c(s))^{-1}=\int^{\infty}_{0}\tau^{\frac{s}{2}-1}(\tau+1)^{-1}d\tau$ for any $0<s<2$.
		
		Define $\tilde{A}=A+divA-|A|^{2}$, for the first term on the right of \eqref{2.4}, we estimate that
		\begin{flalign}
			&\|\langle x\rangle^{\delta-\sigma}V_{x}u\|_{L^{2}(\mathbb{R}^{n})}\nonumber\\
			=&\|\langle x\rangle^{\delta-\sigma}c(s)\int^{\infty}_{0}\tau^{\frac{s}{2}}(\tau-\Delta)^{-1}V_{1}(x)(\tau-\Delta_{A})^{-1}ud\tau\|_{L^{2}(\mathbb{R}^{n})}\nonumber\\
			\leqslant&C\int^{\infty}_{0}\tau^{\frac{s}{2}}\|\langle x\rangle^{\delta-\sigma}(\tau-\Delta)^{-1}V_{1}(x)(\tau-\Delta_{A})^{-1}u\|_{L^{2}(\mathbb{R}^{n})}d\tau\nonumber\\
			\leqslant&C\int^{\infty}_{0}\tau^{\frac{s}{2}}\|\langle x\rangle^{\delta-\sigma}(\tau-\Delta)^{-1}\tilde{A}(\tau-\Delta_{A})^{-1}u\|_{L^{2}(\mathbb{R}^{n})}d\tau\nonumber\\
			+&C\int^{\infty}_{0}\tau^{\frac{s}{2}}\|\langle x\rangle^{\delta-\sigma}(\tau-\Delta)^{-1}A\cdot\nabla(\tau-\Delta_{A})^{-1}u\|_{L^{2}(\mathbb{R}^{n})}d\tau.\label{2.6}
		\end{flalign}	
Because the integral of $\tau$, we divide it into two parts, i.e. $\int_{0}^{\infty}=\int_{0}^{1}+\int_{1}^{\infty}$.\\
$Step\ 1$. When $\tau\in(0,1)$, for the first term on the right of \eqref{2.6}, we have
		\begin{flalign}
			&\|\langle x\rangle^{\delta-\sigma}(\tau-\Delta)^{-1}\tilde{A}(\tau-\Delta_{A})^{-1}u\|_{L^{2}}\nonumber\\
			\leqslant&\|\langle x\rangle^{\delta-\sigma}(\tau-\Delta)^{-1}\langle x\rangle^{-\frac{\beta}{2}}\|_{L^{2}\rightarrow L^{2}}\cdot\|\langle x\rangle^{-\frac{\beta}{2}}(\tau-\Delta_{A})^{-1}u\|_{L^{2}}\nonumber\\
			\leqslant&C\tau^{-1+\frac{\frac{\beta}{2}-\delta+\sigma}{2}}\|\langle x\rangle^{-\frac{\beta}{2}}(R_{0}-R_{A}V_{1}(x)R_{0})u\|_{L^{2}}\nonumber\\
			\leqslant&C\tau^{-1+\frac{\frac{\beta}{2}-\delta+\sigma}{2}}(\|\langle x\rangle^{-\frac{\beta}{2}}R_{0}u\|_{L^{2}}+\|\langle x\rangle^{-\frac{\beta}{2}}R_{A}V_{1}(x)R_{0}u\|_{L^{2}})\nonumber\\
			\leqslant&C\tau^{-1+\frac{\frac{\beta}{2}-\delta+\sigma}{2}}(\|\langle x\rangle^{-\frac{\beta}{2}}R_{0}\langle x\rangle^{\sigma}\cdot \langle x\rangle^{-\sigma}u\|_{L^{2}}+\|\langle x\rangle^{-\frac{\beta}{2}}\nabla R_{0}\langle x\rangle^{\sigma}\cdot \langle x\rangle^{-\sigma}u\|_{L^{2}})\nonumber\\
			\leqslant&C\tau^{-1+\frac{\frac{\beta}{2}-\delta+\sigma}{2}}(\tau^{-1+\frac{\frac{\beta}{2}-\sigma}{2}}\|u\|_{L^{2,-\sigma}}+\|\langle x\rangle^{-\frac{\beta}{2}} R_{0}(-\Delta)^{\frac{\alpha}{2}}\langle x\rangle^{\sigma}\cdot \langle x\rangle^{-\sigma}u\|_{L^{2}})\nonumber\\
			\leqslant&C\tau^{-1+\frac{\frac{\beta}{2}-\delta+\sigma}{2}}(\tau^{-1+\frac{\frac{\beta}{2}-\sigma}{2}}\|u\|_{L^{2,-\sigma}}+\tau^{-1+\frac{\frac{\beta}{2}-\sigma}{2}}\|u\|_{H^{\alpha,-\sigma}})\nonumber\\
			\leqslant&C\tau^{-2+\frac{\beta-\delta}{2}}\|u\|_{H^{\alpha,-\sigma}}.\label{2.7}
		\end{flalign}
For the second term on the right of \eqref{2.6}, we have
		\begin{flalign}
			&\|\langle x\rangle^{\delta-\sigma}(\tau-\Delta)^{-1}A\cdot\nabla(\tau-\Delta_{A})^{-1}u\|_{L^{2}}\nonumber\\
			\leqslant&\|\langle x\rangle^{\delta-\sigma}(\tau-\Delta)^{-1}\langle x\rangle^{-\frac{\beta}{2}}\|_{L^{2}\rightarrow L^{2}}\cdot\|\langle x\rangle^{-\frac{\beta}{2}}\nabla(\tau-\Delta_{A})^{-1}u\|_{L^{2}}\nonumber\\
			\leqslant&C\tau^{-1+\frac{\frac{\beta}{2}-\delta+\sigma}{2}}\|\langle x\rangle^{-\frac{\beta}{2}}\nabla(R_{0}-R_{A}V_{1}(x)R_{0})u\|_{L^{2}}\nonumber\\
			\leqslant&C\tau^{-1+\frac{\frac{\beta}{2}-\delta+\sigma}{2}}(\|\langle x\rangle^{-\frac{\beta}{2}}\nabla R_{0}u\|_{L^{2}}+\|\langle x\rangle^{-\frac{\beta}{2}}\nabla R_{A}V_{1}(x)R_{0}u\|_{L^{2}})\nonumber\\
			\leqslant&C\tau^{-1+\frac{\frac{\beta}{2}-\delta+\sigma}{2}}(\|\langle x\rangle^{-\frac{\beta}{2}}\nabla R_{0}u\|_{L^{2}}+\|\langle x\rangle^{-\frac{\beta}{2}}R_{0}u\|_{L^{2}})\nonumber\\
			\leqslant&C\tau^{-1+\frac{\frac{\beta}{2}-\delta+\sigma}{2}}(\tau^{-1+\frac{\frac{\beta}{2}-\sigma}{2}}\|u\|_{H^{\alpha,-\sigma}}+\tau^{-1+\frac{\frac{\beta}{2}-\sigma}{2}}\|u\|_{L^{2,-\sigma}})\nonumber\\
			\leqslant&C\tau^{-2+\frac{\beta-\delta}{2}}\|u\|_{H^{\alpha,-\sigma}}.\label{2.8}
		\end{flalign}
Substitute \eqref{2.7} and \eqref{2.8} into \eqref{2.6}, we obtain that
		\begin{flalign}
			&\int^{1}_{0}\tau^{\frac{s}{2}}\|\langle x\rangle^{\delta-\sigma}(\tau-\Delta)^{-1}V_{1}(x)(\tau-\Delta_{A})^{-1}u\|_{L^{2}(\mathbb{R}^{n})}d\tau\nonumber\\
			\leqslant&C\int^{1}_{0}\tau^{\frac{s}{2}-2+\frac{\beta-\delta}{2}}d\tau\cdot\|u\|_{H^{\alpha,-\sigma}}\nonumber\\
			\leqslant&C\|u\|_{H^{s,-\sigma}},\label{2.9}
		\end{flalign}
		where the last line of \eqref{2.9} is valid as long as the conditions are met, which is $\beta\geqslant2+\delta-\frac{n}{2}$, $0<\alpha<\min\{1,s\}$.
		
		$Step\ 2$. When $\tau>1$, for the first term on the right of \eqref{2.6}, we have
		\begin{flalign}
			&\|\langle x\rangle^{\delta-\sigma}(\tau-\Delta)^{-1}\tilde{A}(\tau-\Delta_{A})^{-1}u\|_{L^{2}}\nonumber\\
			\leqslant&\|\langle x\rangle^{\delta-\sigma}(\tau-\Delta)^{-1}\langle x\rangle^{-\frac{\beta}{4}}\|_{L^{2}\rightarrow L^{2}}\cdot\|\langle x\rangle^{-\frac{\beta}{2}}\|_{L^{2}\rightarrow L^{2}}\cdot\|\langle x\rangle^{-\frac{\beta}{4}}(\tau-\Delta_{A})^{-1}u\|_{L^{2}}\nonumber\\
			\leqslant&C\tau^{-1+\frac{\frac{\beta}{4}-\delta+\sigma}{2}}\|\langle x\rangle^{-\frac{\beta}{4}}(R_{0}-R_{A}V_{1}(x)R_{0})u\|_{L^{2}}\nonumber\\
			\leqslant&C\tau^{-1+\frac{\frac{\beta}{4}-\delta+\sigma}{2}}(\|\langle x\rangle^{-\frac{\beta}{4}}R_{0}u\|_{L^{2}}+\|\langle x\rangle^{-\frac{\beta}{4}}R_{A}V_{1}(x)R_{0}u\|_{L^{2}})\nonumber\\
			\leqslant&C\tau^{-1+\frac{\frac{\beta}{4}-\delta+\sigma}{2}}(\|\langle x\rangle^{-\frac{\beta}{4}}R_{0}\langle x\rangle^{\sigma}\cdot \langle x\rangle^{-\sigma}u\|_{L^{2}}+\tau^{-\frac{1}{2}}\|\langle x\rangle^{-\frac{\beta}{4}}\nabla R_{0}u\|_{L^{2}})\nonumber\\
			\leqslant&C\tau^{-1+\frac{\frac{\beta}{4}-\delta+\sigma}{2}}(\tau^{-1+\frac{\frac{\beta}{4}-\sigma}{2}}\|u\|_{L^{2,-\sigma}}\nonumber\\
			&+\tau^{-\frac{1}{2}}\|(-\Delta)^{\frac{\alpha}{2}}\langle x\rangle^{-\frac{\beta}{4}}(\Delta)^{\frac{1-\alpha}{2}} R_{0}(-\Delta)^{\frac{\alpha}{2}}\langle x\rangle^{\sigma}\cdot \langle x\rangle^{-\sigma}u\|_{L^{2}})\nonumber\\
			\leqslant&C\tau^{-1+\frac{\frac{\beta}{4}-\delta+\sigma}{2}}(\tau^{-1+\frac{\frac{\beta}{4}-\sigma}{2}}\|u\|_{L^{2,-\sigma}}+\tau^{-\frac{1}{2}+(-1+\frac{\frac{\beta}{4}-\sigma+1}{2})}\|u\|_{H^{\alpha,-\sigma}})\nonumber\\
			\leqslant&C\tau^{-2+\frac{\beta-2\delta}{4}}\|u\|_{H^{\alpha,-\sigma}}.\label{2.10}
		\end{flalign}
For the second term on the right of \eqref{2.6}, we have
		\begin{flalign}
			&\|\langle x\rangle^{\delta-\sigma}(\tau-\Delta)^{-1}A\cdot\nabla(\tau-\Delta_{A})^{-1}u\|_{L^{2}}\nonumber\\
			\leqslant&\|\langle x\rangle^{\delta-\sigma}(\tau-\Delta)^{-1}\langle x\rangle^{-\frac{\beta}{4}}\|_{L^{2}\rightarrow L^{2}}\cdot\|\langle x\rangle^{-\frac{\beta}{2}}\|_{L^{2}\rightarrow L^{2}}\cdot\|\langle x\rangle^{-\frac{\beta}{4}}\nabla(\tau-\Delta_{A})^{-1}u\|_{L^{2}}\nonumber\\
			\leqslant&C\tau^{-1+\frac{\frac{\beta}{4}-\delta+\sigma}{2}}\|\langle x\rangle^{-\frac{\beta}{2}}\nabla(R_{0}-R_{A}V_{1}(x)R_{0})u\|_{L^{2}}\nonumber\\
			\leqslant&C\tau^{-1+\frac{\frac{\beta}{4}-\delta+\sigma}{2}}(\|\langle x\rangle^{-\frac{\beta}{4}}\nabla R_{0}u\|_{L^{2}}+\tau^{-\frac{1}{2}}\|\langle x\rangle^{-\frac{\beta}{4}}\nabla R_{0}u\|_{L^{2}})\nonumber\\
			\leqslant&C\tau^{-2+\frac{\beta-2\delta+2-2\alpha}{4}}\|u\|_{H^{\alpha,-\sigma}}.\label{2.11}
		\end{flalign}
Substitute \eqref{2.10} and \eqref{2.11} into \eqref{2.6}, we obtain that
		\begin{flalign}
			&\int^{\infty}_{1}\tau^{\frac{s}{2}}\|\langle x\rangle^{\delta-\sigma}(\tau-\Delta)^{-1}V_{1}(x)(\tau-\Delta_{A})^{-1}u\|_{L^{2}(\mathbb{R}^{n})}d\tau\nonumber\\
			\leqslant&C\int^{1}_{0}\tau^{\frac{s}{2}-2+\frac{\beta-2\delta+2-2\alpha}{4}}d\tau\cdot\|u\|_{H^{\alpha,-\sigma}}\nonumber\\
			\leqslant&C\|u\|_{H^{s,-\sigma}},\label{2.12}
		\end{flalign}
		where the last line of \eqref{2.12} is valid as long as the conditions are met, which is $\beta<2\delta+2+2\alpha-n$, $0<\alpha<\min\{1,s\}$.
		
		Next, for the second term on the right of \eqref{2.4}, we estimate that
		\begin{flalign}
			&\|\langle x\rangle^{\delta-\sigma}(-\Delta)^{\frac{\alpha}{2}} V_{x}u\|_{L^{2}(\mathbb{R}^{n})}\nonumber\\
			\leqslant&C\int^{\infty}_{0}\tau^{\frac{s}{2}}\|\langle x\rangle^{\delta-\sigma}(-\Delta)^{\frac{\alpha}{2}}(\tau-\Delta)^{-1}V_{1}(x)(\tau-\Delta_{A})^{-1}u\|_{L^{2}(\mathbb{R}^{n})}d\tau\nonumber\\
			\leqslant&C\int^{\infty}_{0}\tau^{\frac{s}{2}}\|\langle x\rangle^{\delta-\sigma}(-\Delta)^{\frac{\alpha}{2}}(\tau-\Delta)^{-1}\tilde{A}(\tau-\Delta_{A})^{-1}u\|_{L^{2}(\mathbb{R}^{n})}d\tau\nonumber\\
			+&C\int^{\infty}_{0}\tau^{\frac{s}{2}}\|\langle x\rangle^{\delta-\sigma}(-\Delta)^{\frac{\alpha}{2}}(\tau-\Delta)^{-1}A\cdot\nabla(\tau-\Delta_{A})^{-1}u\|_{L^{2}(\mathbb{R}^{n})}d\tau.\label{2.13}
		\end{flalign}
$Step\ 3$. When $\tau\in(0,1)$, for the first term on the right of \eqref{2.13}, we have
		\begin{flalign}
			&\|\langle x\rangle^{\delta-\sigma}(-\Delta)^{\frac{\alpha}{2}}(\tau-\Delta)^{-1}\tilde{A}(\tau-\Delta_{A})^{-1}u\|_{L^{2}}\nonumber\\
			\leqslant&\|\langle x\rangle^{\delta-\sigma}(-\Delta)^{\frac{\alpha}{2}}(\tau-\Delta)^{-1}\langle x\rangle^{-\frac{\beta}{2}}\|_{L^{2}\rightarrow L^{2}}\cdot\|\langle x\rangle^{-\frac{\beta}{2}}(\tau-\Delta_{A})^{-1}u\|_{L^{2}}\nonumber\\
			\leqslant&C\tau^{-1+\frac{\frac{\beta}{2}-\delta+\sigma+\alpha}{2}}(\tau^{-1+\frac{\frac{\beta}{2}-\sigma}{2}}\|\langle x\rangle^{-\frac{\beta}{2}}(\tau-\Delta_{A})^{-1}u\|_{L^{2}}\nonumber\\
			\leqslant&C\tau^{-1+\frac{\frac{\beta}{2}-\delta+\sigma+\alpha}{2}}\cdot\tau^{-1+\frac{\frac{\beta}{2}-\sigma}{2}}\|u\|_{H^{\alpha,-\sigma}}\nonumber\\
			\leqslant&C\tau^{-2+\frac{\beta-\delta+\alpha}{2}}\|u\|_{H^{\alpha,-\sigma}}.\label{2.14}
		\end{flalign}
For the second term on the right of \eqref{2.13}, we have
		\begin{flalign}
			&\|\langle x\rangle^{\delta-\sigma}(-\Delta)^{\frac{\alpha}{2}}(\tau-\Delta)^{-1}A\cdot\nabla(\tau-\Delta_{A})^{-1}u\|_{L^{2}}\nonumber\\
			\leqslant&\|\langle x\rangle^{\delta-\sigma}(-\Delta)^{\frac{\alpha}{2}}(\tau-\Delta)^{-1}\langle x\rangle^{-\frac{\beta}{2}}\|_{L^{2}\rightarrow L^{2}}\cdot\|\langle x\rangle^{-\frac{\beta}{2}}\nabla(\tau-\Delta_{A})^{-1}u\|_{L^{2}}\nonumber\\
			\leqslant&C\tau^{-1+\frac{\frac{\beta}{2}-\delta+\sigma+\alpha}{2}}\|\langle x\rangle^{-\frac{\beta}{2}}\nabla(R_{0}-R_{A}V_{1}(x)R_{0})u\|_{L^{2}}\nonumber\\
			\leqslant&C\tau^{-1+\frac{\frac{\beta}{2}-\delta+\sigma+\alpha}{2}}\tau^{-1+\frac{\frac{\beta}{2}-\sigma}{2}}\|u\|_{H^{\alpha,-\sigma}}\nonumber\\
			\leqslant&C\tau^{-2+\frac{\beta-\delta+\alpha}{2}}\|u\|_{H^{\alpha,-\sigma}}.\label{2.15}
		\end{flalign}
Combining \eqref{2.14} and \eqref{2.15}, we obtain that
		\begin{flalign}
			&\int^{1}_{0}\tau^{\frac{s}{2}}\|\langle x\rangle^{\delta-\sigma}(-\Delta)^{\frac{\alpha}{2}}(\tau-\Delta)^{-1}V_{1}(x)(\tau-\Delta_{A})^{-1}u\|_{L^{2}(\mathbb{R}^{n})}d\tau\nonumber\\
			\leqslant&C\int^{1}_{0}\tau^{\frac{s}{2}-2+\frac{\beta-\delta+\alpha}{2}}d\tau\cdot\|u\|_{H^{\alpha,-\sigma}}\nonumber\\
			\leqslant&C\|u\|_{H^{s,-\sigma}},\label{2.16}
		\end{flalign}
		where the last line of \eqref{2.16} is valid as long as the conditions are met, which is $\beta\geqslant1+\delta-\frac{n}{2}$, $0<\alpha<\min\{1,s\}$.
		
		$Step\ 4$. When $\tau>1$, for the first term on the right of \eqref{2.13}, we have
		\begin{flalign}
			&\|\langle x\rangle^{\delta-\sigma}(-\Delta)^{\frac{\alpha}{2}}(\tau-\Delta)^{-1}\tilde{A}(\tau-\Delta_{A})^{-1}u\|_{L^{2}}\nonumber\\
			\leqslant&\|\langle x\rangle^{\delta-\sigma}(-\Delta)^{\frac{\alpha}{2}}(\tau-\Delta)^{-1}\langle x\rangle^{-\frac{\beta}{4}}\|_{L^{2}\rightarrow L^{2}}\cdot\|\langle x\rangle^{-\frac{\beta}{2}}\|_{L^{2}\rightarrow L^{2}}\cdot\|\langle x\rangle^{-\frac{\beta}{4}}(\tau-\Delta_{A})^{-1}u\|_{L^{2}}\nonumber\\
			\leqslant&C\tau^{-1+\frac{\frac{\beta}{4}-\delta+\sigma+\alpha}{2}}\|\langle x\rangle^{-\frac{\beta}{4}}(\tau-\Delta_{A})^{-1}u\|_{L^{2}}\nonumber\\
			\leqslant&C\tau^{-1+\frac{\frac{\beta}{4}-\delta+\sigma+\alpha}{2}}\cdot\tau^{-1+\frac{\frac{\beta}{4}-\sigma}{2}}\|u\|_{H^{\alpha,-\sigma}}\nonumber\\
			\leqslant&C\tau^{-2+\frac{\beta-2\delta+2\alpha}{4}}\|u\|_{H^{\alpha,-\sigma}}.\label{2.17}	
		\end{flalign}
For the second term on the right of \eqref{2.13}, we have
		\begin{flalign}
			&\|\langle x\rangle^{\delta-\sigma}(-\Delta)^{\frac{\alpha}{2}}(\tau-\Delta)^{-1}A\cdot\nabla(\tau-\Delta_{A})^{-1}u\|_{L^{2}}\nonumber\\
			\leqslant&\|\langle x\rangle^{\delta-\sigma}(-\Delta)^{\frac{\alpha}{2}}(\tau-\Delta)^{-1}\langle x\rangle^{-\frac{\beta}{4}}\|_{L^{2}\rightarrow L^{2}}\cdot\|\langle x\rangle^{-\frac{\beta}{2}}\|_{L^{2}\rightarrow L^{2}}\cdot\|\langle x\rangle^{-\frac{\beta}{4}}\nabla(\tau-\Delta_{A})^{-1}u\|_{L^{2}}\nonumber\\
			\leqslant&C\tau^{-1+\frac{\frac{\beta}{4}-\delta+\sigma+\alpha}{2}}\cdot\tau^{-1+\frac{\frac{\beta}{4}-\sigma}{2}}\|\langle x\rangle^{-\frac{\beta}{4}}\nabla R_{0}u\|_{L^{2}}\nonumber\\
			\leqslant&C\tau^{-2+\frac{\beta-2\delta+2\alpha}{4}}\|u\|_{H^{\alpha,-\sigma}}.\label{2.18}
		\end{flalign}
Combining \eqref{2.17} and \eqref{2.18}, we obtain that
		\begin{flalign}
			&\int^{\infty}_{1}\tau^{\frac{s}{2}}\|\langle x\rangle^{\delta-\sigma}(-\Delta)^{\frac{\alpha}{2}}(\tau-\Delta)^{-1}V_{1}(x)(\tau-\Delta_{A})^{-1}u\|_{L^{2}(\mathbb{R}^{n})}d\tau\nonumber\\
			\leqslant&C\int^{1}_{0}\tau^{\frac{s}{2}-2+\frac{\beta-2\delta+2\alpha}{4}}d\tau\cdot\|u\|_{H^{\alpha,-\sigma}}\nonumber\\
			\leqslant&C\|u\|_{H^{s,-\sigma}},\label{2.19}
		\end{flalign}
		where the last line of \eqref{2.19} is valid as long as the conditions are met, which is $\beta<2\delta+2\alpha-n$, $0<\alpha<\min\{1,s\}$.
		
		The estimates of $Step\ 1$-$Step\ 4$  can be obtained
		$$\|\langle x\rangle^{\delta}V_{x}u\|_{H^{\alpha,-\sigma}(\mathbb{R}^{n})}\leqslant C\|u\|_{H^{s,-\sigma}}, \ 0<\alpha<\min\{1,s\}.$$
Furthermore, when $n>3$ with $2k+1<s<2k+2$, $k=1,2,...$, split $s=s-k+k$, using the similar methods as $n\leqslant3$, \eqref{2.4} can also be estimated. In fact, the high-dimensional and low-dimensional processing methods are the same, hence we omit the proof of this part here.
		
		Therefore, the boundedness of operator $$\langle x\rangle^{\delta}V_{x}: H^{s,-\sigma}\rightarrow H^{\alpha,-\sigma},$$
		is completed for $\delta>1$, $\sigma>\frac{1}{2}$ and $0<\alpha<\min\{1,s\}$.
		
		Finally, we shall prove the compactness of operator $V_{x}$. Recall the fact that $\mathbb{A}\circ\mathbb{B}$ is a compact operator if $\mathbb{A}$ is a compact operator and $\mathbb{B}$ is a bounded operator. Write $V_{x}=\langle x\rangle^{-\delta}\cdot\langle x\rangle^{\delta}V_{x}$ with $\delta>1$, and we have proved the boundedness of operator $\langle x\rangle^{\delta}V_{x}$, next the compactness of operator $\langle x\rangle^{-\delta}: H^{\alpha,-\sigma}\rightarrow L^{2,\sigma}$ shall be dealt with. When $\delta>1$, the operator $\langle x\rangle^{-\delta}$ is a multiplication operator, and belongs to Agmon potential class (see \cite{Ag}). Then, the compactness of operator $\langle x\rangle^{-\delta}: H^{\alpha,-\sigma}\rightarrow L^{2,\sigma}$ is obvious.
		
		Combining the boundedness of $\langle x\rangle^{\delta}V_{x}: H^{s,-\sigma}\rightarrow H^{\alpha,-\sigma}$ and the compactness of $\langle x\rangle^{-\delta}: H^{\alpha,-\sigma}\rightarrow L^{2,\sigma}$, we have the compactness of $V_{x}:H^{s,-\sigma}\rightarrow L^{2,\sigma}$. Therefore, the proof of Lemma \ref{lemma2.2} is completed.
	\end{proof}
	
	\subsection{The boundary values of the resolvent}
	Let $R^{s}_{A}$ and $R_{0}^{s}$ be the resolvents of operators $(-\Delta_{A})^{\frac{s}{2}}$ and $(-\Delta)^{\frac{s}{2}}$ respectively. The following lemmas study the resolvent operator $R^{s}_{0}(\lambda\pm i0):=((-\Delta)^{\frac{s}{2}}-(\lambda\pm i0))^{-1}$.
	
	\begin{lemma}[{\cite[Theorem 14.1.7]{H}}]\label{lemma2.3}
		Given $\phi\in C_c^\infty$, and set $\phi(D-\eta)u=\mathscr{F}^{-1}(\phi(\cdot-\eta)\hat{u})$ for $\eta\in\mathbb{R}^n$ and $u\in\mathscr{S}'$. Then we obtain for some $\sigma>\frac{1}{2}$ that
		\begin{equation}\label{2.20}
			\int\|\phi(D-\eta)u\|_{L^{2,\sigma}}d\eta\leqslant C\|u\|_{L^{2,\sigma}},\quad u\in L^{2,\sigma}.
		\end{equation}
		Moreover, if $\|\phi\|_{L^2}>0$, we have
		\begin{equation}\label{2.21}
			\|u\|_{L^{2,-\sigma}}^2\leqslant C\int\|\phi(D-\eta)u\|_{L^{2,-\sigma}}^2d\eta,\quad u\in L_{\mathrm{loc}}^2\cap\mathscr{S}'.
		\end{equation}
	\end{lemma}
	
	Now we start by a localized resolvent estimate for real valued functions.
	
	\begin{lemma}[{\cite[Theorem 14.2.2]{H}}]\label{lemma2.4}
		Let $X\subset\mathbb{R}^n$ be bounded and open, $\phi\in C_c^\infty(X)$, $C_1,C_2>0$, consider all $f(\xi)\in C^2(X;\mathbb{R})$ such that
		\begin{equation}\label{2.22}
			\arrowvert\partial_{j}f(\xi)\arrowvert\geqslant C_1\,\,\text{and}\,\,\arrowvert\partial^\alpha f(\xi)\arrowvert\leqslant C_2,\quad j=1,2,...,\ \xi\in\mathrm{supp}\,\phi,\,\,\arrowvert\alpha\arrowvert\leqslant2.
		\end{equation}
		If $u\in L^{2,\sigma}$, and $\sigma>\frac{1}{2}$, then
		\begin{equation}\label{2.23}
			\mathbb{C}^\pm\ni z\mapsto\mathscr{F}^{-1}((f(\xi)-z)^{-1}\phi\hat{u})\in L^{2,-\sigma}
		\end{equation}
		can be uniquely defined as a weak* continuous function and
		\begin{equation}\label{2.24}
			\|\mathscr{F}^{-1}((f(\xi)-z)^{-1}\phi\hat{u})\|_{L^{2,-\sigma}}\leqslant C\|u\|_{L^{2,\sigma}},\quad u\in L^{2,\sigma},~z\in\mathbb{C}^\pm,
		\end{equation}
		where $\mathbb{C}^\pm=\{z;\pm Imz\geqslant0\}$.
	\end{lemma}
	\begin{proof}
		Using Theorem 14.2.2 in \cite{H}, since $L^{2,\sigma}\subset B$ and $B^{*}\subset L^{2,-\sigma}$ for $\sigma>\frac{1}{2}$, we obtain \eqref{2.24} directly. Hence the proof of Lemma \ref{lemma2.4} is completed.
	\end{proof}
	
	\begin{lemma}\label{lemma2.5}
		Let $X\subset\mathbb{R}^n$ be bounded and open, $\phi\in C_c^\infty(X)$, and $Q_0\in C(X;\mathbb{R})$. Denote $K_1=\overline{\{\xi\in X;~\arrowvert Q_0(\xi)\arrowvert<4\epsilon\}}\cap\mathrm{supp}\,\phi$ for some $\epsilon>0$,  if $Q_0\in C^2(K_1)$ and $\nabla Q_0\neq0$ holds in $K_1$, then there exist $\delta_0$, $C>0$, and $Q\in C(X;\mathbb{R})\cap C^2(K_1)$ such that
		\begin{equation}\label{2.25}
			\|Q-Q_0\|_{C^2(K_1)}+\|Q-Q_0\|_{C(K_2)}<\delta_0,
		\end{equation}
		where $K_2=\overline{\{\xi\in X;~\arrowvert Q_0(\xi)\arrowvert>3\epsilon\}}\cap\mathrm{supp}\,\phi$, we have
		\begin{equation*}
			\|\mathscr{F}^{-1}((Q-\zeta)^{-1}\phi\hat{u})\|_{L^{2,-\sigma}}\leqslant C\|u\|_{L^{2,\sigma}},
		\end{equation*}
		\begin{equation}\label{2.26}
			u\in {L^{2,\sigma}},\, \sigma>\frac{1}{2}, ~\zeta\in\mathbb{C}^\pm,~\arrowvert\zeta\arrowvert<\delta_0.
		\end{equation}
	\end{lemma}
	\begin{proof}
		By the assumption, we first note that $0$ is not a critical value of $Q_0$ in $\mathrm{supp}\,\phi$. This will lead to the fact that $0$ is not a critical value of $Q-\zeta$ in $\mathrm{supp}\,\phi$ as long as $\arrowvert\zeta\arrowvert$ is small and \eqref{2.25} holds.
		
		Take $h\in C_c^\infty(\mathbb{R})$ such that $h(s)=1$ when $\arrowvert s\arrowvert<3\epsilon$ and $h(s)=0$ when $s>4\epsilon$, and denote
		\begin{equation}\label{2.27}
			\phi_1(\xi)=h(Q_0(\xi))\phi(\xi),\quad\phi_2(\xi)=(1-h(Q_0(\xi)))\phi(\xi),
		\end{equation}
		then $\phi=\phi_1+\phi_2$, and $\mathrm{supp}\,\phi_1\subset K_1$. From the property of $Q_{0}$, let's assume that
		\begin{equation*}
			\arrowvert\nabla Q_0(\xi)\arrowvert\geqslant C_1\,\,\text{and}\,\,\arrowvert\partial^\alpha Q_0(\xi)\arrowvert\leqslant C_2,\quad\xi\in K_1,~\arrowvert\alpha\arrowvert\leqslant2.
		\end{equation*}
		Then for $\forall\ \xi\in K_1$, there exists $j(\xi)\in\{1,\cdots,n\}$ such that $\arrowvert\partial_{j(\xi)}Q_0(\xi)\arrowvert\geqslant\frac{C_1}{\sqrt{n}}$, and choosing $\delta_0\leqslant\mbox{$\frac{1}{2\sqrt{n}}$}$
		\begin{equation*}
			\|Q-Q_0\|_{C^2(K_1)}<\delta_0\leqslant\mbox{$\frac{C_1}{2\sqrt{n}}$},
		\end{equation*}
		we have
		\begin{equation*}
			\arrowvert\partial_{j(\xi)}Q(\xi)\arrowvert\geqslant\mbox{$\frac{C_1}{2\sqrt{n}}$}\,\,\text{and}\,\,\arrowvert\partial^\alpha Q(\xi)\arrowvert\leqslant C_2+\mbox{$\frac{C_1}{2\sqrt{n}}$},\quad\xi\in K_1,~\arrowvert\alpha\arrowvert\leqslant2.
		\end{equation*}
		By Lemma \ref{lemma2.4}, we obtain that
		\begin{equation*}
			\|\mathscr{F}^{-1}((Q-\zeta)^{-1}\phi_1\hat{u})\|_{L^{2,-\sigma}}\leqslant C\|u\|_{L^{2,\sigma}},\quad u\in {L^{2,\sigma}},~\zeta\in\mathbb{C}^\pm.
		\end{equation*}
		
		On the other hand, when $\xi\in\mathrm{supp}\,\phi_2\subset K_2$, we have $\arrowvert Q_0(\xi)\arrowvert\geqslant3\epsilon$. Thus when $\arrowvert\zeta\arrowvert<\delta_0\leqslant\epsilon$ and
		\begin{equation*}
			\|Q-Q_0\|_{C(K_2)}<\delta_0,
		\end{equation*}
		we have $\arrowvert Q-\zeta\arrowvert\geqslant\epsilon$ in $\mathrm{supp}\,\phi_2$, and consequently
		\begin{equation*}
			\|\mathscr{F}^{-1}((Q-\zeta)^{-1}\phi_2\hat{u})\|_{L^{2,-\sigma}}\leqslant\|\mathscr{F}^{-1}((Q-\zeta)^{-1}\phi_2\hat{u})\|_{L^2}\leqslant\epsilon^{-1}\|u\|_{L^{2,\sigma}},
		\end{equation*}
		which completes the proof of Lemma \ref{lemma2.5}.
	\end{proof}
	
	\begin{lemma}\label{lemma2.6}
		Given $z_0\in\mathbb{R}\setminus\{0\}$ and $\phi\in C_c^\infty(B(1))$, where $B(1)$ is the unit ball centered on the origin. For $\eta\in\mathbb{R}^n$, we define
		\begin{equation}\label{2.28}
			Q_\eta(\xi)=\frac{\arrowvert\xi+\eta\arrowvert^s-z_0}{\langle\eta\rangle^s},\quad\xi\in\mathbb{R}^n.
		\end{equation}
		Then there exist $\delta_0,C>0$ such that
		$$\|\mathscr{F}^{-1}((Q_\eta-\zeta)^{-1}\phi\hat{u})\|_{L^{2,-\sigma}}\leqslant C\|u\|_{L^{2,\sigma}},$$
		\begin{equation}\label{2.29}
			u\in L^{2,\sigma},~\sigma>\frac{1}{2}, ~\zeta\in\mathbb{C}^\pm,~\arrowvert\zeta\arrowvert<\delta_0.
		\end{equation}
	\end{lemma}
	
	\begin{proof}
		As in the proof of Lemma \ref{lemma2.5}, we show through the following details that when $z_0\neq0$ and $\arrowvert\zeta\arrowvert$ is small, $0$ is not a critical value of $Q-\zeta$ in $\mathrm{supp}\,\phi$.
		
		First consider the set of functions $\mathcal{A}=\{Q_\eta;~\arrowvert\eta\arrowvert>2\}$. Obviously $\mathcal{A}\subset C^\infty(\overline{B(1)})$. Let $\overline{\mathcal{A}}$ be the closure of $\mathcal{A}$ in $C^2(\overline{B(1)})$ and suppose $Q\in\overline{\mathcal{A}}$. If $Q=Q_\eta$ for some $\eta$ with $\arrowvert\eta\arrowvert\geqslant2$, we have
		\begin{equation*}
			\arrowvert\nabla Q(\xi)\arrowvert=\frac{s\arrowvert\xi+\eta\arrowvert^{s-1}}{\langle\eta\rangle^s}\geqslant C_\eta>0,\quad\xi\in\overline{B(1)}.
		\end{equation*}
		If $Q=\mathop{\lim}\limits_{j\rightarrow\infty}Q_{\eta_j}$ in $C^2(\overline{B(1)})$ for some $\eta_j\rightarrow\infty$, by the inequality
		\begin{equation*}
			\left\arrowvert\arrowvert\xi\arrowvert^s-z_0\right\arrowvert+s\arrowvert\xi\arrowvert^{s-1}+1\geqslant C_{z_0,s}\mbox{$(\frac52+\arrowvert\xi\arrowvert^2)^\frac s2$},\quad\xi\in\mathbb{R}^n\setminus\{0\},
		\end{equation*}
		we know for all $j$ that
		\begin{flalign*}
			&\frac{\left\arrowvert\arrowvert\xi+\eta_j\arrowvert^s-z_0\right\arrowvert}{\langle\eta_j\rangle^s}+\frac{s\arrowvert\xi+\eta_j\arrowvert^{s-1}}{\langle\eta_j\rangle^s}+\frac{1}{\langle\eta_j\rangle^s}\nonumber\\
			\geqslant&C_{z_0,s}\left(\frac{\frac52+\arrowvert\xi+\eta_j\arrowvert^2}{1+\arrowvert\eta_j\arrowvert^2}\right)^\frac s2\\ \geqslant&2^{-\frac s2}C_{z_0,s},\quad\xi\in\overline{B(1)},
		\end{flalign*}
		and obtain by sending $j$ to $\infty$ that
		\begin{equation*}
			\arrowvert Q(\xi)\arrowvert+\arrowvert\nabla Q(\xi)\arrowvert\geqslant C>0,\quad\xi\in\mathrm{supp}\,\phi,
		\end{equation*}	
		which also holds in the previous case. Therefore $Q$ satisfies the condition of $Q_0$ in Lemma \ref{lemma2.5}. On the other hand, one checks by the Arzel\`{a}-Ascoli theorem that $\overline{\mathcal{A}}$ is compact in $C^2(\overline{B(1)})$. Now we know \eqref{2.29} is true for $\arrowvert\eta\arrowvert>2$ if we employ a finite covering argument for $\overline{\mathcal{A}}$ with respect to the $C^2$ topology by using Lemma \ref{lemma2.5}.
		
		Next consider $\mathcal{B}=\{Q_\eta;~\arrowvert\eta\arrowvert\leqslant2\}$. If $z_0>0$, we take $\epsilon=\frac{z_0}{8\langle\eta\rangle^s}$, then $\arrowvert Q_\eta(\xi)\arrowvert<4\epsilon$ implies $\arrowvert\xi+\eta\arrowvert^s>\frac{z_0}{2}$. Therefore $Q_\eta\in C^2(K_1)$ where $K_1$ is defined in Lemma \ref{lemma2.5} with $Q_0$ replaced by $Q_\eta$, and
		\begin{equation*}
			\arrowvert\nabla Q_\eta(\xi)\arrowvert=\frac{s\arrowvert\xi+\eta\arrowvert^{s-1}}{\langle\eta\rangle^s}\geqslant C_\eta>0,\quad\xi\in\mathrm{supp}\,\phi.
		\end{equation*}
		This implies
		\begin{equation*}
			\arrowvert Q_\eta(\xi)\arrowvert+\arrowvert\nabla Q_\eta(\xi)\arrowvert\geqslant C_\eta>0,\quad\xi\in\mathrm{supp}\,\phi,
		\end{equation*}
		which is also true if $z_0>0$ for $Q_\eta$ never vanishes. Now we apply Lemma \ref{lemma2.5} with $Q_0$ replaced by $Q_\eta$. For $\delta_0$ obtained in Lemma \ref{lemma2.5}, by the boundedness of $\xi$ and $\eta$, it is easy to see that there exists $\delta>0$ such that $\arrowvert\eta'-\eta\arrowvert<\delta$ implies
		\begin{equation*}
			\|Q_{\eta'}-Q_\eta\|_{C^2(K_1)}+\|Q_{\eta'}-Q_\eta\|_{C(K_2)}<\delta_0.
		\end{equation*}
		Thus a finite covering argument for $\{\eta\in\mathbb{R}^n;~\arrowvert\eta\arrowvert\leqslant2\}$ shows that \eqref{2.29} is also true for $\mathcal{B}$. Now the proof of Lemma \ref{lemma2.6} is completed.
	\end{proof}
	
	It is clear that if $\lambda\in\mathbb{R}\setminus\{0\}$, $\lambda$ is not a critical value of $\arrowvert\xi\arrowvert^s$ which is $C^\infty$ near $\{\arrowvert\xi\arrowvert^s=\lambda\}$, thus by Lemma \ref{lemma2.4}, $R_0^{s}(\lambda\pm i0)u$ is well defined at least when $\hat{u}\in C_c^\infty$, and now we are about to prove the main result of this part.
	
	\begin{lemma}\label{lemma2.7}
		Let $K$ be a closed subset of $\mathbb{C}^+$ or $\mathbb{C}^-$ such that $0\notin K$ and $\mathrm{Re}\,K$ is bounded. Then we have
		$$\|R_0^{s}(z)u\|_{L^{2,-\sigma}}\leqslant C_K\|u\|_{L^{2,\sigma}},$$
		\begin{equation}\label{2.30} u\in\mathscr{F}^{-1}C_c^\infty,~\sigma>\frac{1}{2},~s>0,~z\in K.
		\end{equation}
		Moreover, $R_0^{s}(z)$ maps $L^{2,\sigma}$ into $H^{s,-\sigma}$.
	\end{lemma}
	
	\begin{proof}
		For any $z_0\in\mathbb{R}\setminus\{0\}$, we first prove that \eqref{2.30} is true for $z\in\mathbb{C}^\pm$ with $\arrowvert z-z_0\arrowvert$ sufficiently small. Notice that $\arrowvert z-z_0\arrowvert<\epsilon$ implies $\left\arrowvert\frac{z-z_0}{\langle\eta\rangle^s}\right\arrowvert<\epsilon$ for $\eta\in\mathbb{R}^n$, thus we take $\phi\in C_c^\infty(B(1))$ with $\phi\equiv1$ in $B(\frac12)$, and apply Lemma \ref{lemma2.6} with $\zeta=\frac{z-z_0}{\langle\eta\rangle^s}$ to have
		\begin{flalign*}
			&\left\|\mathscr{F}^{-1}\left(\left(\mbox{$\frac{\arrowvert\cdot+\eta\arrowvert^s-z}{\langle\eta\rangle^s}$}\right)^{-1}\phi\hat{u}\right)\right\|_{L^{2,-\sigma}}\\
			=&\left\|\mathscr{F}^{-1}\left(\left(\mbox{$\frac{\arrowvert\cdot+\eta\arrowvert^s-z_0}{\langle\eta\rangle^s}-\frac{z-z_0}{\langle\eta\rangle^s}$}\right)^{-1}\phi\hat{u}\right)\right\|_{L^{2,-\sigma}}\\
			\leqslant&C\|u\|_{L^{2,\sigma}},\  u\in {L^{2,\sigma}},~\eta\in\mathbb{R}^n,~z\in\mathbb{C}^\pm,~\arrowvert z-z_0\arrowvert<\epsilon,
		\end{flalign*}
		for some $\epsilon>0$. Therefore, if we take $\phi_0\in C_c^\infty(B(\frac12))$ with $\|\phi_0\|_{L^2}>0$, and replace the above $\hat{u}$ with $\phi_0\hat{u}(\cdot+\eta)$ where $u\in\mathscr{F}^{-1}C_c^\infty$, we have $\phi_0\hat{u}(\cdot+\eta)\in \mathscr{F}{L^{2,\sigma}}$ and then for $\eta\in\mathbb{R}^n$ that
		\begin{flalign}
			&\langle\eta\rangle^s\|\phi_0(D-\eta)R_0^{s}(z)u\|_{L^{2,-\sigma}}\nonumber\\
			\leqslant&C\|\phi_0(D-\eta)u\|_{L^{2,\sigma}}\label{2.31}\\
			\leqslant&C\left\|(I-\Delta)^\frac s2\phi_0(D-\eta)u\right\|_{L^{2,\sigma}}\nonumber\\
			=&C\left\|(I-\Delta)^\frac s2\phi(D-\eta)\phi_0(D-\eta)u\right\|_{L^{2,\sigma}}.\nonumber
		\end{flalign}
		Squaring both sides of \eqref{2.31} and integrating on $\mathbb{R}^n$ with respect to $\eta$, \eqref{2.30} when $z\in\mathbb{C}^\pm$ and $\arrowvert z-z_0\arrowvert<\epsilon$ is then a consequence of \eqref{2.20} and \eqref{2.21}.
		
		Now we apply a finite covering to $K\cap\mathbb{R}$ in a $\mathbb{C}^+$- or $\mathbb{C}^-$-neighborhood, and obtain \eqref{2.30} when $\arrowvert\mathrm{Im}\,z\arrowvert<M$ for some $M>0$. The rest when $\arrowvert\mathrm{Im}\,z\arrowvert\geqslant M$ is a result of the trivial estimate
		\begin{equation*}
			\begin{split}
				\|R_0^{s}(z)u\|_{L^2}\leqslant&C\left\|\frac{\langle\cdot\rangle^s\hat{u}}{\arrowvert\cdot\arrowvert^s-\mathrm{Re}\,z-i\mathrm{Im}\,z}\right\|_{L^2}\\
				\leqslant&C_{\mathrm{Re}\,K,M}\|u\|_{L^2},
			\end{split}
		\end{equation*}
		and embedding $L^{2,\sigma}\hookrightarrow L^2\hookrightarrow L^{2,-\sigma}$. The proof of \eqref{2.30} is now complete. That is, the proof of Lemma \ref{lemma2.7} is completed.
	\end{proof}
	Now we have finished the preparation before construction. Let's construct the new distorted Fourier transforms.
	
	\begin{definition}\label{definition2.2}
		Let $A(x)$ satisfy hypothesis \eqref{1.1}, $s>0$. We define $\Lambda$ be the set of $\lambda\in\mathbb{R}^{+}\setminus\{0\}$ such that the equation
		\begin{equation*}
			(I+V_{x}R_{0}^{s}(\lambda\pm i0))u=0
		\end{equation*}
		has a non-trivial solution $u\in L^{2,\sigma}$ with $\sigma>\frac{1}{2}$.
	\end{definition}
	
	Let $\{E_\lambda\}$ be the spectral family of $\mathscr{H}$, and $$\tilde{\Lambda}=\Lambda\cup\{0\},\ E^d=\int_{\tilde{\Lambda}}dE_\lambda,\ E^c=\int_{\mathbb{R}\setminus\tilde{\Lambda}}dE_\lambda,$$
	where $E^{d}$ and $E^{c}$ respectively denote the point spectrum and continuous spectrum of $\mathscr{H}$, we have
	\begin{lemma}\label{lemma2.8}
		If $z\in\mathbb{C}^\pm\setminus\tilde{\Lambda}$, then $(I+V_{x}R_{0}^{s}(z))^{-1}$ exist and is continuous in $L^{2,\sigma}$. Moreover, the maps
		\begin{equation}\label{2.32}
			\mathbb{C}^\pm\setminus\tilde{\Lambda}\ni z\mapsto(I+V_{x}R_{0}^{s}(z))^{-1}u\in L^{2,\sigma}
		\end{equation}
		are continuous if $u\in L^{2,\sigma}$,  $\sigma>\frac{1}{2}$.
	\end{lemma}
	\begin{proof}
		When $\mathrm{Im}\, z\neq0$, for $u\in L^{2,\sigma}$ we have $R_0^{s}(z)u\in H^s$, and it is easy to deduce that
		\begin{equation}\label{2.33}
			R_0^{s}(z)u=R(z)(I+V_{x}R_0^{s}(z))u,
		\end{equation}
		where $I+V_{x}R_0^{s}(z)$ can be interpreted as an operator in $L^{2,\sigma}$. Thus $I+V_{x}R_0^{s}(z)$ has kernel $\{0\}$, which is also true when $z\in\mathbb{C}^\pm\setminus\tilde{\Lambda}$. By Fredholm theory, $(I+V_{x}R_0^{s}(z))^{-1}$ is bounded in $L^{2,\sigma}$ and strongly continuous in $z\in\mathbb{C}^\pm\setminus\tilde{\Lambda}$, where the proof needs the facts that  $\mathbb{C}^\pm\setminus\{0\}\ni z\mapsto V_{x}R_0^{s}(z)u\in L^{2,\sigma}$ is continuous when  $u\in L^{2,\sigma}$ and $\{V_{x}R_0^{s}(z)u;~\|u\|_{L^{2,\sigma}}\leqslant1,\, z\in K\}$  is pre-compact in $L^{2,\sigma}$ if $K$ is bounded (see Lemma \ref{lemma2.2}), which obviously holds. Therefore the proof of Lemma \ref{lemma2.8} is completed.
	\end{proof}
	
	\begin{definition}\label{definition2.3}
		If $u\in L^{2,\sigma}$ with $\sigma>\frac{1}{2}$, then the $L^{2}$ functions defined by
		\begin{equation}\label{2.34}
			F_\pm^{A}u(\xi)=\mathscr{F}((I+V_{x}R_0^{s}(\lambda\pm i0 ))^{-1}u)(\xi)
		\end{equation}
		almost everywhere in $M_\lambda=\{\xi\in\mathbb{R}^n;~\arrowvert\xi\arrowvert^s=\lambda\}$ is called the distorted Fourier transforms of $u$.
	\end{definition}
	
	\begin{lemma}\label{lemma2.9} (Intertwining property)
		
		If $u\in L^{2,\sigma}$ with $\sigma>\frac{1}{2}$, we have
		\begin{equation}\label{2.35}
			F_\pm^{A}e^{it\mathscr{H}}u=e^{it\arrowvert\xi\arrowvert^s}F_\pm^{A} u,\quad s>0,\  t\in\mathbb{R}.
		\end{equation}
	\end{lemma}
	
	\begin{proof}
		We first consider $u\in\mathscr{F}^{-1}C_c^\infty(\mathbb{R}^n\setminus\{0\})$. For any $\lambda\in\mathbb{R}^{+}\setminus{\tilde{\Lambda}}$, we have $(\mathscr{H}_0-\lambda)u\in\mathscr{S}$ and thus $R_0^{s}(\lambda\pm i0)(\mathscr{H}_0-\lambda)u=u$ by weak $*$ continuity of $R_0$ and dominated convergence. It follows that
		\begin{equation*}
			(\mathscr{H}-\lambda)u=(I+V_{x}R_0^{s}(\lambda\pm i0))(\mathscr{H}_0-\lambda)u,
		\end{equation*}
		which implies
		\begin{equation*}
			(F_\pm^{A}(\mathscr{H}-\lambda)u)(\xi)=(\arrowvert\xi\arrowvert^s-\lambda)\hat{u}(\xi)=0,\quad when\ \xi\in M_\lambda,
		\end{equation*}
		and this just means
		\begin{equation}\label{2.36}
			F_\pm^{A}\mathscr{H}u=\arrowvert\xi\arrowvert^sF_\pm^{A}u.
		\end{equation}
		Since $\mathscr{F}^{-1}C_c^\infty(\mathbb{R}^n\setminus\{0\})$ is dense in $H^s$ and $\mathscr{H}$ is closed, \eqref{2.36} holds when $u\in H^s$, which in the meanwhile shows that $F_\pm^{A} u\in\mathrm{Dom}(\arrowvert\cdot\arrowvert^s)$. Now we can differentiate $e^{-it\arrowvert\xi\arrowvert^s}F_\pm^{A} e^{it\mathscr{H}}u$ in $t$ when $u\in H^s$, and show that the derivative is $0$, therefore \eqref{2.35} holds for $H^s$ is dense in $L^2$. That is, the proof of Lemma \ref{lemma2.9} is completed.
	\end{proof}
	\section{The properties of wave operators}
	This section introduces the existence and asymptotic completeness of wave operators.
	\subsection{The existence of wave operators}
	This subsection is mainly to prove the existence of wave operators. We have been proved that $V_{x}$ is a short range perturbation of $\mathscr{H}_{0}$ in Lemma \ref{lemma2.2}, which is a sufficient condition to prove the existence of wave operators.
	
	We define the wave operators $W_{\pm}$ by the strong $L^{2}$ limits
	\begin{equation*}
		W_{\pm}u=\lim_{t\rightarrow\pm\infty}e^{it\mathscr{H}}e^{-it\mathscr{H}_{0}}u,\ u\in L^{2}(\mathbb{R}^{n}).
	\end{equation*}
	Using the properties of $V_{x}$ and space decomposition, we prove the existence of the wave operators as follows.
	\begin{lemma}\label{lemma3.1}
		If $A(x)$ satisfies hypothesis \eqref{1.1}, then $W_{\pm}$ exist and isometric for any $s>0$. Therefore, the following intertwining property holds:
		\begin{equation}\label{3.1}
			e^{it\mathscr{H}}W_{\pm}=W_{\pm}e^{it\mathscr{H}_{0}},\ t\in\mathbb{R}.
		\end{equation}
	\end{lemma}
	\begin{proof}
		Since $e^{it\mathscr{H}}e^{-it\mathscr{H}_{0}}$ is unitary, the existence and isometry of the limits $W_{\pm}$, and the intertwining property follow if we can show the existence in a dense subset of $L^2$. By Cook's method (see \cite{RS2} in P20), notice that $\mathscr{F}^{-1}C_c^\infty(\mathbb{R}^n\setminus\{0\})$ is a dense subset of $H^s$, it suffices to show for any $u\in\mathscr{F}^{-1}C_c^\infty(\mathbb{R}^n\setminus\{0\})$ that
		\begin{equation}\label{3.2}
			\int_{\arrowvert t\arrowvert>1}\|V_{x}e^{-it\mathscr{H}_0}u\|_{L^2}dt<\infty.
		\end{equation}
		In fact, if $u\in\mathscr{S}$, then $e^{-it\mathscr{H}_0}u$ is a $C^{\infty}$ function of $t$ with values in $\mathscr{S}$, which proves the existence of the derivative
		$$\frac{d}{dt}(e^{it\mathscr{H}}e^{-it\mathscr{H}_0}u)=e^{it\mathscr{H}}(i\mathscr{H}-i\mathscr{H}_{0})e^{-it\mathscr{H}_0}u=ie^{it\mathscr{H}}V_{x}e^{-it\mathscr{H}_0}$$
		in $L^{2}$. The limit \eqref{1.3} exist if the integral is integrable in \eqref{3.2}.
		
		First notice that there exist positive numbers $r_{1}$ and $r_{2}$ such that
		\begin{equation*}
			0<r_{1}<\left\arrowvert D(\arrowvert\xi\arrowvert^s)\right\arrowvert<r_{2},\quad \xi\in\mathrm{supp}\,\hat{u},
		\end{equation*}
		where $D_{j}=\frac{\partial}{i\partial\xi_{j}}$, $j=1,...,n$. We now estimate
		$$e^{-it\mathscr{H}_0}u=(2\pi)^{-n}\int_{\mathbb{R}^{n}} e^{i(x\cdot\xi-t\arrowvert\xi\arrowvert^s)}\hat{u}(\xi)d\xi.$$
		Since
		\begin{equation*}
			\partial_x^\alpha e^{-it\mathscr{H}_0}u(x)=(2\pi)^{-n}\int_{\mathbb{R}^{n}} e^{i(x\cdot\xi-t\arrowvert\xi\arrowvert^s)}\hat{u}(\xi)(i\xi)^\alpha d\xi,
		\end{equation*}
		and integration by part using the facts that $\xi\mapsto\frac{x\cdot\xi-t\arrowvert\xi\arrowvert^s}{\arrowvert x\arrowvert+\arrowvert t\arrowvert}$ has uniformly bounded derivatives in $\mathrm{supp}\,\hat{u}$, and has gradient uniformly bounded from below there. Hence, we obtain that
		$$\arrowvert\partial_x^\alpha e^{-it\mathscr{H}_0}u(x)\arrowvert\leqslant C_{N,\alpha}(\arrowvert x\arrowvert+\arrowvert t\arrowvert)^{-N},$$
		\begin{equation}\label{3.3}
			\alpha\in\mathbb{N}_0^n,\, N=1,2,...;\arrowvert x\arrowvert<\frac{r_{1}\arrowvert t\arrowvert}{2}~\text{or}~\arrowvert x\arrowvert>2r_{2}\arrowvert t\arrowvert.
		\end{equation}
		
		Now take $\varphi\in C_c^\infty(\mathbb{R}^{n}\backslash{0})$ such that $\varphi(x)=1$ when $\frac{r_{1}}{2}<\arrowvert x\arrowvert<2r_{2}$ and $\varphi(x)=0$ when $\arrowvert x\arrowvert<\frac{r_{1}}{4}$ or $\arrowvert x\arrowvert>4r_{2}$. Let
		\begin{equation*}
			u^1(x)=(1-\varphi(\frac{x}{t}))e^{-it\mathscr{H}_0}u(x),\quad u^2(x)=\varphi(\frac{x}{t})e^{-it\mathscr{H}_0}u(x).
		\end{equation*}
		
		For $u^1$, we have
		\begin{flalign}
			\|V_{x}u^1\|_{L^2}\leqslant&\|V_{x}u^1\|_{L^{2,\sigma}}\nonumber\\
			\leqslant&\|V_{x}\|_{H^{s,-\sigma}\rightarrow L^{2,\sigma}}\cdot\|u^1\|_{H^{s,-\sigma}}\nonumber\\
			\leqslant&\|V_{x}\|_{H^{s,-\sigma}\rightarrow L^{2,\sigma}}\cdot\|(I-\Delta)^{\frac{s}{2}}u^1\|_{L^{2,-\sigma}}\nonumber\\
			\leqslant&\|V_{x}\|_{H^{s,-\sigma}\rightarrow L^{2,\sigma}}\cdot\|(I-\Delta)^{\frac{s}{2}}u^1\|_{L^{2}}.\label{3.4}
		\end{flalign}
		We may take integer $m\geqslant\frac s2$, and use \eqref{3.3} together with Leibniz' formula to deduce
		\begin{equation*}
			\|(I-\Delta)^{\frac{s}{2}}u^1\|_{L^2}\leqslant\|(I-\Delta)^m(1-\varphi(\frac{x}{t}))e^{-it\mathscr{H}_0}u\|_{L^2}\leqslant C\arrowvert t\arrowvert^{-2}.
		\end{equation*}
		Therefore
		\begin{equation}\label{3.5}
			\int_{\arrowvert t\arrowvert>1}\|V_{x}u^1\|_{L^2}dt\leqslant C\int_{\arrowvert t\arrowvert>1}\arrowvert t\arrowvert^{-2}dt<\infty.
		\end{equation}
		
		For $u^2$, we have $\frac{r_{1}\arrowvert t\arrowvert}{4}<\arrowvert x\arrowvert<4r_{2}\arrowvert t\arrowvert$ in $\mathrm{supp}\,u^2$, then we have
		\begin{flalign}
			\|V_{x}u^2\|_{L^2(\mathbb{R}^{n})}\leqslant&C\|\langle x\rangle^{-\delta}u^2\|_{H^{s}(\mathbb{R}^{n})}\nonumber\\
			\leqslant&C\|\langle x\rangle^{-\delta}u^2\|_{L^{2}(\mathbb{R}^{n})}+C\|\langle x\rangle^{-\delta}(-\Delta)^{\frac{s}{2}}u^2\|_{L^{2}(\mathbb{R}^{n})}\nonumber\\
			\leqslant&C|t|^{-\delta}\|u^2\|_{H^{s}(\mathbb{R}^{n})},\delta>1,\label{3.6}
		\end{flalign}
where the first inequality holds because that we can use the same method as the proof of Lemma \ref{lemma2.2}. By taking integer $m\geqslant\frac s2$ and using Leibniz' formula, we also know from $\hat{u}\in C_c^\infty$ that
		\begin{equation*}
			\|u^2\|_{H^{s,-\sigma}}\leqslant\|(I-\Delta)^m\varphi(\frac{x}{t})e^{-it\mathscr{H}_0}u\|_{L^2}\leqslant C,
		\end{equation*}	
		and therefore
		\begin{equation*}
			\|V_{x}u^2\|_{L^2}\leqslant C\arrowvert t\arrowvert^{-\delta}.
		\end{equation*}
Finally, it follows from property of $V_{x}$ that
		\begin{flalign}
			\int_{\arrowvert t\arrowvert>1}\|V_{x}u^2\|_{L^2}\leqslant C\int_{\arrowvert t\arrowvert>1}\arrowvert t\arrowvert^{-\delta}dt<\infty.\label{3.7}
		\end{flalign}
		This completes the proof of \eqref{3.2} and the existence of wave operators. Hence the proof of Lemma \ref{lemma3.1} is completed.
	\end{proof}
	
	\subsection{The asymptotic completeness of the wave operators}
	
	In this subsection, we main study the scattering of fractional magnetic Schr\"{o}dinger operators $(-\Delta_{A})^{\frac{s}{2}}$. The highlight of this section is using the relationship between the distorted Fourier transforms and wave operators, we prove the asymptotic completeness of the wave operators $W_{\pm}$.
	
	Now, we will prove the asymptotic completeness of $W_{\pm}$ by applying the intertwining property of the distorted Fourier transforms in Lemma \ref{lemma2.9} and the properties of wave operators in Lemma \ref{lemma3.1}.
	\begin{lemma}\label{lemma3.2}
		$F_\pm^{A}W_\pm=\mathscr{F}$ holds in $L^2$.
	\end{lemma}
	\begin{proof}
		Here, we just show the proof of  $F_-^{A}W_-=\mathscr{F}$, the other $F_+^{A}W_+=\mathscr{F}$ can use the same method to prove. If $u\in\mathscr{F}^{-1}C_c^\infty(\mathbb{R}^n\setminus\{0\})$, i.e. $e^{-it\mathscr{H}_0}u\in\mathscr{F}^{-1}C_c^\infty(\mathbb{R}^n\setminus\{0\})$ is infinitely differentiable in $L^2$, and we can differentiate $e^{it\mathscr{H}}e^{-it\mathscr{H}_{0}}u$ to get
		\begin{equation}\label{3.8}
			W_-u=u-\int_{-\infty}^0ie^{it\mathscr{H}}V_{x}e^{-it\mathscr{H}_0}udt.
		\end{equation}
		Indeed,
		$$\frac{d}{dt}(e^{it\mathscr{H}}e^{-it\mathscr{H}_0}u)=ie^{it\mathscr{H}}V_{x}e^{-it\mathscr{H}_0}u,$$ and integration by parts, this proves \eqref{3.8}. Since $F_-^{A}$ is $L^2$ continuous, we use the intertwining property \eqref{2.35} and the fact that the integral in \eqref{3.8} is absolutely convergent in $L^2$ by \eqref{3.2} in Lemma \ref{lemma3.1}, to deduce that
		\begin{flalign}
			F_-^{A}W_-u=&F_-^{A}(u-\int_{-\infty}^0ie^{it\mathscr{H}}V_{x}e^{-it\mathscr{H}_0}udt)\nonumber\\
			=&F_-^{A}u-\int_{-\infty}^0F_-^{A}(e^{it\mathscr{H}}iV_{x}e^{-it\mathscr{H}_0}u)dt\nonumber\\
			=&F_-^{A}u-\int_{-\infty}^0e^{it\arrowvert\xi\arrowvert^{s}}F_-^{A}(iV_{x}e^{-it\mathscr{H}_0}u)dt\nonumber\\
			=&F_-^{A}u-\lim_{\epsilon\downarrow0}\int_{-\infty}^0e^{\epsilon t+it\arrowvert\xi\arrowvert^s}F_-^{A}(iV_{x}e^{-it\mathscr{H}_0}u)dt.\label{3.9}
		\end{flalign}
		
		Now we take any $\phi\in C_c^\infty(\mathbb{R}^n\setminus\mathop{\bigcup}\limits_{\lambda\in\tilde{\Lambda}}M_\lambda)$, and assume that $\mathrm{supp}\,\phi\subset I\times\mathbb{S}^{n-1}$, where compact $I\subset\mathbb{R}^{+}\setminus\tilde{\Lambda}$. Note that $\lambda\geqslant c>0$ if $\lambda\in I$. We have
		\begin{flalign}
			&\left\langle\int_{-\infty}^0e^{\epsilon t+it\arrowvert\xi\arrowvert^s}F_-^{A}(iV_{x}e^{-it\mathscr{H}_0}u)dt,\phi\right\rangle\nonumber\\
			=&\int_{-\infty}^0\left\langle e^{\epsilon t+it\arrowvert\xi\arrowvert^s}F_-^{A}(iV_{x}e^{-it\mathscr{H}_0}u),\phi\right\rangle dt\nonumber\\
			=&\int_{-\infty}^0dt\int_{\lambda^\frac1s\in I}\frac{d\lambda}{s\lambda^\frac{s-1}{s}}\int_{M_\lambda}F_-^{A}(e^{\epsilon t+it\lambda}iV_{x}e^{-it\mathscr{H}_0}u)\bar{\phi}dS.\label{3.10}
		\end{flalign}
		Since $\lambda\in\mathbb{R}^{+}\setminus\tilde{\Lambda}$, by Lemma \ref{lemma2.8} we have
		\begin{flalign}
			&\|F_-^{A}(e^{\epsilon t+it\lambda}iV_{x}e^{-it\mathscr{H}_0}u)\|_{L^2(M_\lambda)}\nonumber\\
			\leqslant&C(\lambda)e^{\epsilon t}\|(I+V_{x}R_0^{s}(\lambda-i0))^{-1}V_{x}e^{-it\mathscr{H}_0}u\|_{L^{2,\sigma}}\nonumber\\
			\leqslant&C'(\lambda)e^{\epsilon t}\|(I+V_{x}R_0^{s}(\lambda-i0))^{-1}\|_{L^{2,\sigma}\rightarrow L^{2,\sigma}}\nonumber\\
			&\cdot\|V_{x}\|_{H^{s,-\sigma}\rightarrow L^{2,\sigma}}\cdot\|u\|_{H^{s,-\sigma}}\nonumber\\
			\leqslant&C''(\lambda)e^{\epsilon t},\label{3.11}
		\end{flalign}
		where $C''(\lambda)$ is locally bounded.
		
		Next we use Fubini's theorem to exchange the integrations and obtain that
		\begin{flalign}
			&\left\langle\int_{-\infty}^0e^{\epsilon t+it\arrowvert\xi\arrowvert^s}F_-^{A}(iV_{x}e^{-it\mathscr{H}_0}u)dt,\phi\right\rangle\nonumber\\
			=&\int_{\lambda^\frac1s\in I}\frac{d\lambda}{s\lambda^\frac{s-1}{s}}\int_{-\infty}^0dt\int_{M_\lambda}F_-^{A}(e^{\epsilon t+it\lambda}iV_{x}e^{-it\mathscr{H}_0}u)\bar{\phi}dS\nonumber\\
			=&\int_{\lambda^\frac1s\in I}\frac{d\lambda}{s\lambda^\frac{s-1}{s}}\int_{M_\lambda}\left(\int_{-\infty}^0F_-^{A}(e^{\epsilon t+it\lambda}iV_{x}e^{-it\mathscr{H}_0}u)\arrowvert_{M_\lambda}dt\right)\bar{\phi}dS,\label{3.12}
		\end{flalign}
		where the last line also comes from \eqref{3.11}, and $F_-^{A}(e^{\epsilon t+it\lambda}iV_{x}e^{-it\mathscr{H}_0}u)\arrowvert_{M_\lambda}$ denotes the $L^2$ trace on $M_\lambda$. For fixed $\lambda$, the second estimate of \eqref{3.11} implies that $\int_{-\infty}^0e^{\epsilon t+it\lambda}iV_{x}e^{-it\mathscr{H}_{0}}udt$ is absolutely convergent in $L^{2,\sigma}$; further, the map $u\mapsto F_-^{A}u\arrowvert_{M_\lambda}$ is continuous from $L^{2,\sigma}$ to $L^2(M_\lambda)$. We thus conclude that
		\begin{equation}\label{3.13}
			\int_{-\infty}^0F_-^{A}(e^{\epsilon t+it\lambda}iV_{x}e^{-it\mathscr{H}_0}u)\arrowvert_{M_\lambda}dt=F_-^{A}(\int_{-\infty}^0e^{\epsilon t+it\lambda}iV_{x}e^{-it\mathscr{H}_0}udt)\arrowvert_{M_\lambda}.
		\end{equation}
		Therefore,
		\begin{flalign}
			&\left\langle\int_{-\infty}^0e^{\epsilon t+it\arrowvert\cdot\arrowvert^s}F_-^{A}(iV_{x}e^{-it\mathscr{H}_0}u)dt,\phi\right\rangle\nonumber\\
			=&\int_{\lambda^\frac1s\in I}\frac{d\lambda}{s\lambda^\frac{s-1}{s}}\int_{M_\lambda}F_-^{A}\left(\mbox{$\int_{-\infty}^0e^{\epsilon t+it\lambda}iV_{x}e^{-it\mathscr{H}_0}udt$}\right)\bar{\phi}dS.\label{3.14}
		\end{flalign}
		
		Also notice that $V_{x}$ is continuous from $H^s$ to $L^2$, we have
		\begin{flalign}
			&\left\langle\int_{-\infty}^0e^{\epsilon t+it\arrowvert\cdot\arrowvert^s}F_-^{A}(iV_{x}e^{-it\mathscr{H}_0}u)dt,\phi\right\rangle\nonumber\\
			=&\int_{\lambda^\frac1s\in I}\frac{d\lambda}{s\lambda^\frac{s-1}{s}}\int_{M_\lambda}F_-^{A}V_{x}\left(\mbox{$\int_{-\infty}^0e^{\epsilon t+it\lambda}ie^{-it\mathscr{H}_0}udt$}\right)\bar{\phi}dS\nonumber\\
			=&\int_{\lambda^\frac1s\in I}\frac{d\lambda}{s\lambda^\frac{s-1}{s}}\int_{M_\lambda}(F_-^{A}V_{x}R_0^{s}(\lambda-i\epsilon)u)\bar{\phi}dS.\label{3.15}
		\end{flalign}
		
		Combining \eqref{3.9} and \eqref{3.15}, we obtain that
		\begin{flalign}
			&\left\arrowvert\langle F_-^{A}W_-u-\hat{u},\phi\rangle\right\arrowvert\nonumber\\
			\leqslant&\lim_{\epsilon\downarrow0}\int_{\lambda^\frac1s\in I}\frac{1}{s\lambda^\frac{s-1}{s}}\arrowvert\int_{M_\lambda}\mathscr{F}((I+V_{x}R_0^{s}(\lambda-i0))^{-1}\nonumber\\
			&(I+V_{x}R_0^{s}(\lambda-i\epsilon))u-u)\bar{\phi}dS\arrowvert d\lambda\nonumber\\
			\leqslant&\lim_{\epsilon\downarrow0}C_{I,\phi}\int_{\lambda^\frac1s\in I}\|(I+V_{x}R_0^{s}(\lambda-i0))^{-1}\nonumber\\
			&\cdot(I+V_{x}R_0^{s}(\lambda-i\epsilon))u-u\|_{L^{2,\sigma}}d\lambda\nonumber\\
			=&0,\label{3.16}
		\end{flalign}
		for $(I+V_{x}R_0^{s}(\lambda-i0))^{-1}(I+V_{x}R_0^{s}(\lambda-i\epsilon))u$ is continuous both in $\epsilon$ and in $\lambda$ by Lemma \ref{lemma2.8}. Finally, since $\mathscr{F}^{-1}C_c^\infty(\mathbb{R}^n\setminus\{0\})$ and $C_c^\infty(\mathbb{R}^n\setminus\mathop{\bigcup}\limits_{\lambda\in\tilde{\Lambda}}M_\lambda)$ are both dense in $L^2$, hence the proof of Lemma \ref{lemma3.2} is completed.
	\end{proof}
	
	\section{Absence of embedded eigenvalues for the Schr\"{o}dinger operator}
	In this section, we will discuss the properties of spectrum of operator $\mathscr{H}$ in the following theorem.

	\begin{theorem}\label{theorem4.1}
		Assume that $u\in L^{2}(\mathbb{R}^{n})$ is a solution of the equation
		$$(-\Delta_{A})^{\frac{s}{2}}u=\lambda u,$$
		where $\lambda>0$, $0<s<2$. Then there absence of embedded eigenvalues for the Schr\"{o}dinger operator $(-\Delta_{A})^{\frac{s}{2}}$.
	\end{theorem}
	\begin{proof}
From the assumption, let $\lambda_{0}\neq0$ is an eigenvalue of the operator $(-\Delta_{A})^{\frac{s}{2}}$ with $0<s<2$, then we have  	
\begin{flalign}\label{4.1.1}
(-\Delta_{A})^{\frac{s}{2}}u=\lambda_{0}u.	
\end{flalign}
Combining $u\in L^{2}(\mathbb{R}^{n})$ and the identity \eqref{4.1.1}, applying the iterative method to \eqref{4.1.1}, we obtain $u\in D(-\Delta_{A})$. We can roughly describe this domain space. In fact, we know that $\|u\|^{2}_{H^{s}}=\|u\|^{2}_{L^{2}}+\|(-\Delta)^{\frac{s}{2}}u\|^{2}_{L^{2}}$, $\|u\|_{\dot{H}^{s}}$ is equivalent to $\|(-\Delta_{A})^{\frac{s}{2}}u\|_{L^{2}}$ with $0<s<\frac{n}{2}$, and  $\|(I+(-\Delta_{A}))^{\frac{s}{2}}u\|_{L^{2}}\leqslant C\|u\|_{H^{s}}$ with any $s>0$, then from \eqref{4.1.1} we conclude $H^{2s}(\mathbb{R}^{n})\subset D(-\Delta_{A})$. 	 	
		
Since $0<s<2$, we have $1<\frac{2}{s}<\infty$. Now according to the value of $\frac{2}{s}$, we discuss the situation as follows. When $\frac{2}{s}=m_{1}$ is an integer, we have 		
\begin{flalign}\label{4.1.2}		
(-\Delta_{A})u=(-\Delta_{A})^{\frac{s}{2}\cdot m_{1}}u=(-\Delta_{A})^{\frac{s}{2}}...(-\Delta_{A})^{\frac{s}{2}}u=\lambda_{0}^{m_{1}}u=\lambda_{0}^{\frac{2}{s}}u.
\end{flalign}		
When $\frac{2}{s}$ is not an integer, we can write $\frac{2}{s}=m_{2}+\alpha_{1}$ with $m_{2}=\lceil\frac{2}{s}\rceil$ means an integer not greater than $\frac{2}{s}$ and $\alpha_{1}\in(0,1)$ is not an integer. Then we obtain that 
\begin{flalign}\label{4.1.3}			
(-\Delta_{A})u=(-\Delta_{A})^{\frac{s}{2}\cdot(m_{2}+\alpha_{1})}u=(-\Delta_{A})^{\frac{s\alpha_{1}}{2}}(-\Delta_{A})^{\frac{s}{2}\cdot m_{2}}u=\lambda_{0}^{m_{2}}(-\Delta_{A})^{\frac{s\alpha_{1}}{2}}u,		
\end{flalign}			
recall
\begin{flalign*}
&(-\Delta_{A})^{\frac{s\alpha_{1}}{2}}u=((-\Delta_{A})^{\frac{s}{2}})^{\frac{2\alpha_{1}}{2}}u\\
=&(-\Delta_{A})^{\frac{s}{2}}c(\alpha_{1})\int_{0}^{\infty}\tau^{\frac{2\alpha_{1}}{2}-1}(\tau+(-\Delta_{A})^{\frac{s}{2}})^{-1}ud\tau\\
=&(-\Delta_{A})^{\frac{s}{2}}\lambda_{0}^{\frac{2\alpha_{1}}{2}-1}u=\lambda_{0}^{\alpha_{1}}u,
\end{flalign*}
substitute it into \eqref{4.1.3} to get
\begin{flalign}\label{4.1.4}			
(-\Delta_{A})u=\lambda_{0}^{m_{2}+\alpha_{1}}u=\lambda_{0}^{\frac{2}{s}}u.
\end{flalign}
Therefore, from \eqref{4.1.2} and \eqref{4.1.4}, we obtain that there exist  $\lambda_{0}^{\frac{2}{s}}\neq0$ is an eigenvalue of the operator $-\Delta_{A}$. However, we know that $-\Delta_{A}$ have no possible eigenvalue in \cite{KT}, then the assumption does not hold. That is to say, there are no embedded eigenvalues for the operator $(-\Delta_{A})^{\frac{s}{2}}$. Then the proof of Theorem \ref{theorem4.1} is completed.
\end{proof}
	
Now, combining these results above, we can prove Theorem \ref{theorem1.1}.
\begin{proof} [\textbf{Proof of Theorem \ref{theorem1.1}}]
On the one hand, Lemma \ref{lemma3.1} implies that wave operators exist. On the other hand, since $\mathscr{H}_{0}$ has only absolute continuous spectrum, a general fact (e.g. see \cite[P531]{K3}) is that $Ran(W_{\pm})\subset\mathscr{H}_{ac}$. Now Lemma \ref{lemma2.9} and Lemma \ref{lemma3.2} indicate that $E^{c}L^{2}=Ran(W_{\pm})$, therefore $E^{c}L^{2}\subset\mathscr{H}_{ac}$. Notice that $E^cL^2\supset\mathscr{H}_\mathrm{ac}\oplus\mathscr{H}_\mathrm{sc}$, hence $\mathscr{H}_\mathrm{sc}=\emptyset$, and $Ran(W_{+})=Ran(W_{-})=\mathscr{H}_{ac}$. This implies that the proof of \eqref{1.4} is completed, that is to say, we have the wave operators are asymptotical complete. Combining conclusion in Theorem \ref{theorem4.1}, therefore the proof of Theorem \ref{theorem1.1} is completed.
\end{proof}

	\vspace{2em}
	\phantomsection
		
\end{document}